\documentclass[11pt]{amsart}
\usepackage[margin=1in,letterpaper,portrait]{geometry}
\usepackage{ytableau}
\ytableausetup{notabloids}
\ytableausetup{centertableaux}
\usepackage{latexsym,amsmath,amssymb,verbatim, tikz}
\usepackage{graphicx}
\usepackage{hyperref}

\theoremstyle{plain}
   
   \newtheorem{theorem}{Theorem}[section]
   \newtheorem{proposition}[theorem]{Proposition}
   
   \newtheorem{lemma}[theorem]{Lemma}
   \newtheorem{corollary}[theorem]{Corollary}
   
   \newtheorem{problem}[theorem]{Problem}
\theoremstyle{definition}
   \newtheorem{definition}[theorem]{Definition}
   \newtheorem{defn}[theorem]{Definition}
   \newtheorem{example}[theorem]{Example}

   \newtheorem{remark}[theorem]{Remark}

\numberwithin{equation}{section}

\newcommand\comp[2]{\alpha{(#2,#1)}}
\newcommand\ccomp[2]{\alpha^\cyc{(#2,#1)}}

\newcommand{\ZZ}{{\mathbb {Z}}}

\newcommand{\Cyl}{\operatorname{Cyl}}
\newcommand{\Tor}{\operatorname{Tor}}

\newcommand{\Des}{{\operatorname{Des}}}
\newcommand{\cDes}{{\operatorname{cDes}}}
\newcommand{\cyc}{{\operatorname{cyc}}}
\newcommand{\cdes}{{\operatorname{cdes}}}
\newcommand{\des}{{\operatorname{des}}}
\newcommand{\ch}{{\operatorname{ch}}}

\newcommand{\SYT}{{\operatorname{SYT}}}

\newcommand{\cc}{{\operatorname{cc}_n}}

\newcommand{\cs}{{\tilde{s}}}

\newcommand{\xx}{{\mathbf{x}}}
\newcommand{\ttt}{{\mathbf{t}}}
\newcommand{\symm}{{\mathfrak{S}}}

\newcommand{\OOO}{{\mathcal{O}}}
\newcommand{\TTT}{{\mathcal{T}}}

\newcommand{\SteinbergTorus}{{\widetilde{\Delta}}}

\newcommand{\wcDes}{{\cDes_*}}
\newcommand{\wcdes}{{\cdes_*}}

\newlength{\mysizetiny}
\newlength{\mysizesmall}
\newlength{\mysize}
\newlength{\mysizelarge}
\begin{document}

\title[On cyclic descents for tableaux]
      {On cyclic descents for tableaux}

\author{Ron M.\ Adin}
\address{Department of Mathematics, Bar-Ilan University, Ramat-Gan 52900, Israel}
\email{radin@math.biu.ac.il}
\author{Victor Reiner}
\address{School of Mathematics, University of Minnesota,
Minneapolis, MN 55455, USA}
\email{reiner@math.umn.edu}
\author{Yuval Roichman}
\address{Department of Mathematics, Bar-Ilan University, Ramat-Gan 52900, Israel}
\email{yuvalr@math.biu.ac.il}

\date{November 22, 2017}

\thanks{First and third authors partially supported by a MISTI MIT-Israel grant.
Second author partially supported by NSF grant DMS-1601961.}

\begin{abstract}
The notion of descent set, for permutations as well as for standard Young tableaux (SYT), is classical.
Cellini introduced a natural notion of {\em cyclic descent set} for permutations, and Rhoades introduced such a notion for SYT --- but only for rectangular shapes.
In this work we define {\em cyclic extensions} of descent sets in a general context, 
and prove existence and essential uniqueness for SYT of almost all shapes. 
The proof applies nonnegativity properties of Postnikov's toric Schur polynomials, providing a new interpretation of certain Gromov-Witten invariants.
%
%
%
\end{abstract}


\maketitle
\tableofcontents

\section{Introduction}

For a permutation $\pi = [\pi_1, \ldots, \pi_n]$ 
in the symmetric group $\symm_n$ on $n$ letters, one defines its {\em descent set} as
\[
\Des(\pi) := \{1 \le i \le n-1 \,:\, \pi_i > \pi_{i+1} \}
\quad \subseteq [n-1],
\]
where $[m]:=\{1,2,\ldots,m\}$.  For example, $\Des([2,1,4,5,3])=\{1,4\}$. On the other hand, its
{\em cyclic descent set} was defined by Cellini~\cite{Cellini} as
\begin{equation}
\label{e.cellini}
\cDes(\pi) := \{1 \leq i \leq n \,:\, \pi_i > \pi_{i+1} \}
\quad \subseteq [n],
\end{equation}
with the convention $\pi_{n+1}:=\pi_1$.  For example, $\cDes([2,1,4,5,3])=\{1,4,5\}$.
This cyclic descent set was further studied by Dilks, Petersen, Stembridge~\cite{DPS} and others.
It has the following important properties.
Consider the two $\ZZ$-actions, on $\symm_n$ and on the power set of $[n]$, in which the generator $p$ of $\ZZ$ acts by 
\[
\begin{array}{rcl}
[\pi_1,\pi_2,\ldots,\pi_{n-1},\pi_n]&\overset{p}{\longmapsto}&
[\pi_n,\pi_1,\pi_2,\ldots,\pi_{n-1}], \\
\{i_1,\ldots,i_k\}&\overset{p}{\longmapsto}&\{i_1+1,\ldots,i_k+1\} \bmod{n}.
\end{array}
\]
Then for every permutation $\pi$, one has these three properties:
\begin{eqnarray}
\label{restriction-property}
&\cDes(\pi) \cap [n-1] = \Des(\pi)&\quad \text{(extension)} \\
\label{equivariance-property}
&\cDes(p(\pi))  = p(\cDes(\pi)) &\quad \text{(equivariance)} \\
\label{non-Escher-property}
&\varnothing \subsetneq \cDes(\pi) \subsetneq [n] &\quad \text{(non-Escher)} 
\end{eqnarray}
The term {\em non-Escher} refers to M.\ C.\ Escher's drawing ``Ascending and Descending'', which paradoxically depicts the impossible cases 
$\cDes(\pi)=\varnothing$ and $\cDes(\pi)=[n]$.

\medskip
There is also an established notion of descent set for 
{\em standard (Young) tableau} $T$ of a skew shape $\lambda/\mu$:
\[
\Des(T) := \{1 \le i \le n-1 \,:\,
            i+1 \text{ appears in a lower row of }T\text{ than }i\}
\quad \subseteq [n-1].
\]
For example, this standard Young tableau $T$ of shape $\lambda/\mu=(4,3,2)/(1,1)$ 
has $\Des(T)=\{2,3,5\}$:
\[
\ytableausetup{boxsize=\mysize}
\begin{ytableau}
\none& 1 & 2 & 7 \\
\none& 3& 5 \\
   4 &6 \\
\end{ytableau}
\]
For the special case of standard tableaux $T$ of {\em rectangular} shapes,
Rhoades \cite[Lemma 3.3]{Rhoades} introduced a notion of 
{\em cyclic descent set} $\cDes(T)$, having the same properties
\eqref{restriction-property}, \eqref{equivariance-property} and \eqref{non-Escher-property}
with respect to the $\ZZ$-action in which the generator $p$ acts on tableaux via 
Sch\"utzenberger's {\em jeu-de-taquin promotion} operator. 
A similar concept of $\cDes(T)$ and accompanying action $p$ was 
introduced for two-row partition shapes and certain other skew shapes 
(see Subsection~\ref{known-examples-subsection} for the list) 
in~\cite{AER, ER1}, 
and used there to answer Schur positivity questions. 

\medskip

Our first main result is a necessary and sufficient condition for the existence of a cyclic extension $\cDes$ of the descent map $\Des$ on the set $\SYT(\lambda/\mu)$ of standard Young tableaux of shape $\lambda/\mu$,
with an accompanying $\ZZ$-action on $\SYT(\lambda/\mu)$ via an operator $p$,
satisfying properties 
\eqref{restriction-property}, \eqref{equivariance-property} and \eqref{non-Escher-property}.
In this story, a special role is played by the skew shapes known as
{\em ribbons} (connected skew shapes containing no $2 \times 2$ rectangle),
and in particular {\em hooks} (straight ribbon shapes, namely $\lambda=(n-k,1^k)$ for $k=0,1,\ldots,n-1$).
Early versions of~\cite{AER} and~\cite{ER1} conjectured the following result. 

\begin{theorem}\label{conj1}
Let $\lambda/\mu$ be a skew shape.
The descent map $\Des$ on $\SYT(\lambda/\mu)$ has a cyclic extension $(\cDes,p)$ if and only if $\lambda/\mu$ is not a connected ribbon.
Furthermore, for all $J \subseteq [n]$, all such cyclic extensions share the same cardinalities $\#\cDes^{-1}(J)$.
\end{theorem}

Our strategy for proving Theorem~\ref{conj1} is inspired by a result of Gessel \cite[Theorem 7]{Gessel} that we recall here.
For a subset $J=\{j_1  < \ldots < j_t\} \subseteq [n-1]$, the composition (of $n$)
\begin{equation}\label{composition-of-set}
\comp{n}{J}:= (j_1,j_2-j_1,j_3-j_2,\ldots,j_t-j_{t-1},n-j_t)
\end{equation}
defines a {\em connected ribbon} having the entries of $\comp{n}{J}$ as row lengths, and thus an associated {\em (skew) ribbon Schur function} 
\begin{equation}\label{ribbon-skew-as-alternating-sum}
s_{\comp{n}{J}}
:= \sum_{\varnothing \subseteq I \subseteq J}
(-1)^{\#(J \setminus I)} h_{\comp{n}{I}}
\end{equation}
with the following property: 
for any skew shape $\lambda/\mu$, the
descent map $\Des : \SYT(\lambda/\mu) \longrightarrow 2^{[n-1]}$ 
has fiber sizes given by
\begin{equation}\label{Gessel-ribbon-descent-fact}
\# \Des^{-1}(J) =
\langle s_{\lambda/\mu}, s_{\comp{n}{J}} \rangle
\qquad (\forall J \subseteq [n-1]),
\end{equation}
where $\langle -,-\rangle$ is the usual inner product on symmetric functions. 

By analogy, for a subset $\varnothing \ne J = \{ j_1 < j_2 < \ldots < j_t \}
\subseteq [n]$
we define the corresponding {\em cyclic composition} of $n$  as
\begin{equation}\label{cyclic-composition-of-set}
\ccomp{n}{J} := (j_2-j_1, \ldots, j_t - j_{t-1}, j_1 + n - j_t),
\end{equation}
with $\ccomp{n}{J} := (n)$ when $J = \{j_1\}$;
note that $\ccomp{n}{\varnothing}$ is not defined.
The corresponding {\em affine (or cyclic) ribbon Schur function} is
then defined as
\begin{equation}\label{affine-ribbon-Schur-function}
\cs_{\ccomp{n}{J}} := \sum_{\varnothing \neq I \subseteq J}
(-1)^{\#(J \setminus I)} h_{\ccomp{n}{I}}.
\end{equation}
We then collect enough properties of 
this function
to show that there must exist 
a map $\cDes: \SYT(\lambda/\mu) \rightarrow 2^{[n]}$
and a $\ZZ$-action $p$ on $\SYT(\lambda/\mu)$,
as in Theorem~\ref{conj1}, 
such that fiber sizes are given by
\begin{equation}\label{cDes-fiber-sizes-as-inner-product}
\# \cDes^{-1}(J) =
\langle s_{\lambda/\mu}, \cs_{\ccomp{n}{J}} \rangle 
\qquad (\forall\ \varnothing \subsetneq J \subsetneq [n]).
\end{equation}
See Corollary~\ref{cor:fiber-sizes-as-inner-product} below.
The nonnegativity of this inner product when $\lambda/\mu$ is not a connected ribbon ultimately relies on relating
$\cs_{\ccomp{n}{J}}$ to a special case of Postnikov's {\it toric Schur polynomials},
with their interpretation in terms of {\em Gromov-Witten invariants}
for Grassmannians \cite{Postnikov}.

\medskip
We also compare the distribution of $\cDes$ on $\SYT(\lambda)$
to the distribution of $\cDes$ on $\symm_n$.
Recall \cite[Theorem 3.1.1 and \S 5.6 Ex.\ 22(a)]{Sagan} that the {\it Robinson-Schensted correspondence} is a bijection
between $\symm_n$ and the set of pairs of standard Young tableaux 
of the same shape $\lambda$ (and size $n$), 
having the property that if $w \mapsto (P,Q)$ then $\Des(w)=\Des(Q)$.
Consequently
\[
\sum_{w \in \symm_n} \ttt^{\Des(w)}
= \sum_{ \lambda \vdash n } f^\lambda \sum_{T \in \SYT(\lambda)} \ttt^{\Des(T)}.
\]
Here $\ttt^S:=\prod_{i \in S}t_i$ for   $S \subseteq \{1,2,\ldots\}$, while
$\lambda \vdash n$ means $\lambda$ is a partition of $n$,
and $f^\lambda:=\#\SYT(\lambda)$.
Note that Theorem~\ref{conj1} implies that any {\em non-hook} shape $\lambda$, as well as any disconnected skew shape $\lambda/\mu$, 
will have
$\sum_{T \in \SYT(\lambda/\mu)} \ttt^{\cDes(T)}$ well-defined and
independent of the choice of cyclic extension $(\cDes,p)$ for $\Des$ on $\SYT(\lambda)$.
We then have the following second main result.

\begin{theorem}\label{conj2}
For any $n \ge 2$
\[
\sum_{w \in \symm_n} \ttt^{\cDes(w)}
= \sum_{\substack{\text{non-hook}\\ \lambda \vdash n}}
f^\lambda \sum_{T \in \SYT(\lambda)} \ttt^{\cDes(T)}
\quad + \quad \sum_{k=1}^{n-1} \binom{n-2}{k-1} 
 \sum_{T \in \SYT((1^k)\oplus (n-k))} \ttt^{\cDes(T)},
\]
where $\cDes$ is defined on $\symm_n$ by Cellini's formula \eqref{e.cellini} and on standard Young tableaux (of the relevant shapes) as in Theorem~\ref{conj1}.
\end{theorem}
\noindent
The {\em direct sum} operation $\lambda \oplus \mu$ in the last summation
denotes a skew shape having the diagram of $\lambda$ strictly southwest
of the diagram for $\mu$, with no rows or columns in common.  For example, when
$(n,k)=(7,2)$, 
\[
(1^k)\oplus (n-k)
\quad = \quad (1^2) \oplus (5)
\quad = \quad
\ytableausetup{boxsize=\mysizesmall}
\ydiagram{1,1}
\,\,\, \oplus \,\,\,
\ydiagram{5}
\quad = \quad
\ydiagram{1+5,1,1}
\]





The rest of the paper is structured as follows.
In Section~\ref{sec:definitions} we define and study the abstract notion of cyclic extension of a descent map.
Section~\ref{sec:Postnikov} introduces affine ribbon Schur functions, and compares them to Postnikov's cylindric Schur functions and toric Schur polynomials.
Sections~\ref{sec:conj1-proof} and~\ref{sec:conj2-proof} contain proofs of Theorems \ref{conj1} and \ref{conj2}, respectively.
The multivariate distribution of $\cDes$ on $\symm_n$, as in Theorem~\ref{conj2}, is studied further in Section~\ref{sec:symmetric-group}.
Finally, Section \ref{sec:remarks-and-questions} 
contains various remarks and open questions.

\section*{Acknowledgments}
The authors thank Sergi Elizalde,  Ira Gessel,  Alex Postnikov and Richard Stanley for stimulating discussions, and Connor Ahlbach, Anders Bj\"orner, Alessandro Iraci,  Oliver Pechenik, Kyle Petersen, Jim Propp, Brendon Rhoades, Josh Swanson and Michelle Wachs 
for helpful comments.

\section{Cyclic descents:  definition, examples, and basic properties}
\label{sec:definitions}

\subsection{Definition}

Let us begin by formalizing the concept of a cyclic extension.
Recall the bijection $p: 2^{[n]} \longrightarrow 2^{[n]}$ induced by the cyclic shift $i \mapsto i + 1 \pmod n$, for all $i \in [n]$.

\begin{definition}
\label{def:cDes}
Let $\TTT$ be a finite set. A {\em descent map} is any map
$\Des: \TTT \longrightarrow 2^{[n-1]}$. 
A {\em cyclic extension} of $\Des$ is
a pair $(\cDes,p)$, where 
$\cDes: \TTT \longrightarrow 2^{[n]}$ is a map 
and $p: \TTT \longrightarrow \TTT$ is a bijection,
satisfying the following axioms:  for all $T$ in  $\TTT$,
\[
\begin{array}{rl}
\text{(extension)}   & \cDes(T) \cap [n-1] = \Des(T),\\
\text{(equivariance)}& \cDes(p(T))  = p(\cDes(T)),\\
\text{(non-Escher)}  & \varnothing \subsetneq \cDes(T) \subsetneq [n].\\
\end{array}
\]
\end{definition}

The non-Escher axiom will be important for the uniqueness of the cyclic extension; see Subsection~\ref{subsec:Escher} for a discussion of the consequences of omitting this assumption.

\subsection{Known examples}
\label{known-examples-subsection}

Cyclic extensions of descent maps have been given previously in several cases:

\begin{enumerate}
\item[$\bullet$]
For $\TTT = \symm_n$, the descent set $\Des(\pi)$ of a permutation $\pi$ was described in the Introduction, as was Cellini's original cyclic extension $(\cDes,p)$.
Note that $n \ge 2$ is required for the non-Escher property.

\item[$\bullet$]
More generally, given any (strict) {\em composition} $\alpha$ of $n$,
that is, an ordered sequence of positive integers $\alpha=(\alpha_1,\ldots,\alpha_t)$
with $\sum_i \alpha_i=n$, define the associated {\em horizontal strip} skew shape
\[
\alpha^\oplus :=(\alpha_1) \oplus (\alpha_2) \oplus \cdots \oplus (\alpha_t)
\]
whose rows, from southwest to northeast, have sizes $\alpha_1,\ldots,\alpha_t$. 
For $T$ in $\TTT=\SYT(\alpha^\oplus)$, we define
\[
\cDes(T):=\{1\le i\le n \,:\, i+1 \text{ is in a lower row than } i\},
\]
where $n+1$ is interpreted as $1$,
as well as a bijection $p: \SYT(\alpha^\oplus) \rightarrow \SYT(\alpha^\oplus)$
which first replaces each entry $j$ of $T$ by $j+1 \pmod n$ and then
re-orders each row to make it left-to-right increasing.
One can check that this $(\cDes,p)$ is a cyclic extension of $\Des$,
with $t \ge 2$ required for the non-Escher property.
For example, when $\alpha=(3,4,2)$
(and $n = 9$), one has the following standard tableaux $T$
of shape $\alpha^\oplus$:
\[
\begin{array}{rcl}
\ytableausetup{boxsize=\mysize}
T = \, 
\begin{ytableau}
 \none & \none & \none & \none & \none & 
 \none & \none & 3     & 9 \\
 \none & \none & \none & 1     & 5     & 
 7     & 8 \\
 2     & 4     & 6 \\
\end{ytableau}
& \overset{p}{\longmapsto} &
p(T) = \,
\begin{ytableau}
 \none & \none & \none & \none & \none & 
 \none & \none & 1     & 4 \\
 \none & \none & \none & 2     & 6     & 
 8     & 9 \\
 3     & 5     & 7 \\
\end{ytableau} 
\\
& & 
\\
\cDes(T) = \{1,3,5,9\}
& \overset{p}{\longmapsto} &
\cDes(p(T)) = \{1,2,4,6\}
\end{array}
\]
This generalizes the case of $(\cDes,p)$ on  $\symm_n$, since for $\alpha=(1^n)=(1,1,\ldots,1)$ one has a bijection $\symm_n \rightarrow \SYT(\alpha^\oplus)$ which sends a permutation $w$ to the tableau whose entries are $w^{-1}(1),\ldots,w^{-1}(n)$ read from southwest to northeast; 
e.g., for $n=5$,
\[
\ytableausetup{boxsize=\mysize}
w=[5,3,1,4,2] \quad \longmapsto \quad
\begin{ytableau}
\none& \none&  \none& \none&1\\
\none& \none&  \none& 4\\
\none& \none&  2\\
\none& 5\\
  3 
\end{ytableau}
\]
This bijection maps Cellini's cyclic extension $(\cDes,p)$ on $\symm_n$ to
the one on $\SYT(\alpha^\oplus)$ for $\alpha=(1,1,\ldots,1)$ defined above\footnote{This cyclic descent map can further be 
generalized  to {\em strips}, which are the disconnected shapes each of whose connected components consisting of either one row or one column; see  \cite{AER}.}.

\item[$\bullet$]
Let $\TTT=\SYT(\lambda)$ with $\lambda=(a^b)$ of {\em rectangular} shape, e.g. 
\[
\lambda = (5^3) = \,
\ytableausetup{boxsize=\mysizesmall}
\ydiagram{5,5,5}
\]
Consider the usual notion of descent set $\Des(T)$ on standard tableaux, as in the Introduction.
As mentioned earlier, Rhoades \cite[Lemma 3.3]{Rhoades} showed 
that Sch\"utzenberger's jeu-de-taquin promotion operation $p$ provides
a cyclic extension $(\cDes,p)$.
Again, we require $a, b \ge 2$ for the non-Escher property.

\item[$\bullet$]
Let $\TTT=\SYT(\lambda)$ with $\lambda$ of {\em hook plus internal corner} shape, namely $\lambda = (n-2-k,2,1^k)$ with $0 \le k \le n-4$, e.g. 
\[
\lambda = (8,2,1,1,1) = \,
\ytableausetup{boxsize=\mysizesmall}
\ydiagram{8,2,1,1,1,}
\]
There is a unique
cyclic descent map $\cDes$, defined as follows:
by the extension property, it suffices to specify when $n \in \cDes(T)$,
and one decrees this to hold if and only if the entry $T_{2,2}-1$ 
lies strictly west of $T_{2,2}$ (namely, in the first column of $T$);
see \cite{AER} for more details. 
Note that for most shapes in this family there are several possible shifting bijections $p$, so that the cyclic extension $(\cDes,p)$ is not unique.

\item[$\bullet$]
Let $\TTT=\SYT(\lambda)$ with $\lambda$ of {\em two-row} shape,
namely $\lambda=(n-k,k)$ with $2 \le k\le n/2$, 
e.g. 
\[
\lambda = (8,3) = \,
\ytableausetup{boxsize=\mysizesmall}
\ydiagram{8,3}
\]
There exists a cyclic extension of $\Des$ defined as follows; see \cite{AER}.
Decree that $n\in\cDes(T)$ if and only if both
\begin{itemize}
\item[$-$] the last two entries in the second row of $T$ are consecutive, and 
\item[$-$] for every $1<i<k$,  $T_{2,i-1} > T_{1,i}$.
\end{itemize}

\item[$\bullet$]
For any nonempty partition $\lambda\vdash n-1$, the partition $\lambda \oplus (1)$, e.g.
\[
\lambda=(4,3,1) \oplus (1) = \,
\ytableausetup{boxsize=\mysizesmall}
\ydiagram{4+1,4,3,1}
\]
has an explicit cyclic extension of $\Des$ described in \cite{ER1}.  This
case was, in fact, our original motivation, and the question of existence of a cyclic extension of $\Des$ on $\SYT(\lambda/\mu)$ appears there as \cite[Problem 5.5]{ER1}.
\end{enumerate}

\subsection{Existence and uniqueness of cyclic extensions}

\begin{lemma}\label{t.cyclic_extension}
Fix a set $\TTT$ and a map $\Des: \TTT \to 2^{[n-1]}$.

\begin{enumerate}
\item[$(i)$] 
Any cyclic extension $(\cDes,p)$ of $\Des$ has fiber sizes 
$m(J) := \#\cDes^{-1}(J)$
(for $J \subseteq [n]$) satisfying
\begin{itemize}
\item[$(a)$]
$m(J) \ge 0$ for all $J$, with $m(\varnothing)=m([n])=0$;
\item[$(b)$]
$m(J)=m(p(J))$ for all $J$; and
\item[$(c)$]
$m(J) + m(J \sqcup \{n\}) = \Des^{-1}(J)$ for all $J \subseteq [n-1]$.
\end{itemize}

\item[$(ii)$]
Conversely, if 
$(m(J))_{J \subseteq [n]}$ are integers
satisfying conditions (a),(b),(c) above,
then there exists at least one cyclic extension $(\cDes,p)$ of $\Des$ satisfying $\#\cDes^{-1}(J) = m(J)$ for all $J \subseteq [n]$.

\item[$(iii)$] 
Under the same hypotheses as in (ii), 
the fiber size $m(J) = \#\cDes^{-1}(J)$ is uniquely determined, for every subset
$\varnothing \ne J = \{j_1 < \ldots < j_t\} \subseteq [n]$, by the formula
\begin{equation}
\label{cDes-fiber-sizes-formula}
m(J) = \#\cDes^{-1}(J)
= \sum_{i=1}^{t} (-1)^{i-1} \#\Des^{-1}(\{j_{i+1} - j_i, \ldots, j_t - j_i\}),
\end{equation}
interpreted for $t = 1$ as
\[
m(\{j_1\}) = \#\cDes^{-1}(\{j_1\})
= \#\Des^{-1}(\varnothing).
\]
Of course, $m(\varnothing) = \#\cDes^{-1}(\varnothing) =
0$.
\end{enumerate}

\end{lemma}

\noindent
Note that Lemma~\ref{t.cyclic_extension} does {\em not} assert uniqueness of the bijection
$p: \TTT \rightarrow \TTT$, which can fail,
as in the following example.

\begin{example}
\label{non-unique-orbits-example}
When $\lambda=(3,2,1)$, $\#\SYT(\lambda) = 16$.
Lemma~\ref{t.cyclic_extension}(iii)
determines the following $\cDes$ distribution: each of the sets 
\[
\{1,2,4\}, \{1,2,5\}, \{1,3,4\}, 
\{1,3,6\}, \{1,4,5\}, \{1,4,6\},
\]
\[ 
\{2,3,5\}, \{2,3,6\}, \{2,4,5\}, 
\{2,5,6\}, \{3,4,6\}, \{3,5,6\}
\]
appears once (as a value of $\cDes$), and each of the sets
\[
\{1,3,5\}, \{2,4,6\}
\]
appears twice.
Thus there are two $p$-orbits of size $6$, 
namely this one (with entries of $\cDes(T)$ in bold)
\[
\ytableausetup{boxsize=\mysize}
\cdots \overset{p}{\longmapsto}
\begin{ytableau}
1 & 2 & \mathbf{3} \\
4 & \mathbf{5} &\none\\
\mathbf{6} & \none&\none
\end{ytableau}
\overset{p}{\longmapsto}
\begin{ytableau}
\mathbf{1} & 3 & \mathbf{4} \\
2 & \mathbf{6} &\none\\
5 & \none&\none
\end{ytableau}
\overset{p}{\longmapsto}
\begin{ytableau}
\mathbf{1} & 4 & \mathbf{5} \\
\mathbf{2} & 6 &\none\\
3 & \none&\none
\end{ytableau}
\overset{p}{\longmapsto}
\begin{ytableau}
1 & \mathbf{2} & \mathbf{6} \\
\mathbf{3} & 5 &\none\\
4 & \none&\none
\end{ytableau}
\overset{p}{\longmapsto}
\begin{ytableau}
\mathbf{1} & \mathbf{3} & 6 \\
2 & \mathbf{4} &\none\\
5 & \none&\none
\end{ytableau}
\overset{p}{\longmapsto}
\begin{ytableau}
1 & \mathbf{2} & \mathbf{4} \\
3 & \mathbf{5} &\none\\
6 & \none&\none
\end{ytableau}
\overset{p}{\longmapsto} \cdots
\]
and this one
\[
\cdots \overset{p}{\longmapsto}
\begin{ytableau}
1 & 2 & \mathbf{3} \\
\mathbf{4} & \mathbf{6} &\none\\
5 & \none&\none
\end{ytableau}
\overset{p}{\longmapsto}
\begin{ytableau}
\mathbf{1} & 3 & \mathbf{4} \\
2 & \mathbf{5} &\none\\
6 & \none&\none
\end{ytableau}
\overset{p}{\longmapsto}
\begin{ytableau}
1 & \mathbf{2} & \mathbf{5} \\
3 & 4 &\none\\
\mathbf{6} & \none&\none
\end{ytableau}
\overset{p}{\longmapsto}
\begin{ytableau}
\mathbf{1} & \mathbf{3} & \mathbf{6} \\
2 & 5 &\none\\
4 & \none&\none
\end{ytableau}
\overset{p}{\longmapsto}
\begin{ytableau}
\mathbf{1} & \mathbf{4} & 6 \\
\mathbf{2} & 5 &\none\\
3 & \none&\none
\end{ytableau}
\overset{p}{\longmapsto}
\begin{ytableau}
1 & \mathbf{2} & \mathbf{5} \\
\mathbf{3} & 6 &\none\\
4 & \none&\none
\end{ytableau}
\overset{p}{\longmapsto} \cdots
\]
The remaining four tableaux can either 
form one $p$-orbit of size four, or two $p$-orbits of size two (in two distinct orders each):
\[
\left\{
\begin{ytableau}
\mathbf{1} & \mathbf{3} & \mathbf{5} \\
2 & 4 &\none\\
6 & \none&\none
\end{ytableau}
\quad , \quad
\begin{ytableau}
1 & \mathbf{2} & \mathbf{4} \\
3 & \mathbf{6} &\none\\
5 & \none&\none
\end{ytableau}
\quad , \quad
\begin{ytableau}
\mathbf{1} & \mathbf{3} & \mathbf{5} \\
2 & 6 &\none\\
4 & \none&\none
\end{ytableau}
\quad , \quad
\begin{ytableau}
1 & \mathbf{2} & \mathbf{6} \\
3 & \mathbf{4} &\none\\
5 & \none&\none
\end{ytableau}
\right\}
\]
\end{example}

\begin{proof}[Proof of Lemma~\ref{t.cyclic_extension}(i).]
The nonnegativity property (a) is obvious,
with $m(\varnothing) = m([n]) = 0$ following from the non-Escher axiom,
while properties (b) and (c) are consequences of the equivariance and extension axioms in Definition~\ref{def:cDes}.
\end{proof}

\begin{proof}[Proof of Lemma~\ref{t.cyclic_extension}(ii).]

Let $(m(J))_{J \subseteq [n]}$ be integers satisfying conditions (a),(b),(c) above.
For each $J \subseteq [n-1]$ choose, arbitrarily, a subset $\TTT_J^0$ of $\TTT_J := \Des^{-1}(J)$ of size $\#\TTT_J^0 = m(J)$, and denote 
$\TTT_J^1 := \TTT_J \setminus \TTT_J^0$; by condition (c), $\#\TTT_J^1 = m(J \sqcup \{n\})$.
By construction,
\[
\TTT = \bigsqcup_{J \subseteq [n-1]} (\TTT_J^0 \sqcup \TTT_J^1).
\]
Define $\cDes: \TTT \longrightarrow 2^{[n]}$ by
\[
\cDes(T) := 
\begin{cases}
J, & \text{if } T \in \TTT_J^0; \\
J \sqcup \{n\}, & \text{if } T \in \TTT_J^1.
\end{cases}
\]
Note that this map satisfies $\#\cDes^{-1}(J) = m(J)$ for all $J \subseteq [n]$,
as well as the extension and non-Escher axioms of Definition~\ref{def:cDes}.
It remains to define a bijection $p: \TTT \longrightarrow \TTT$ so that the equivariance axiom is satisfied as well.

To this end, consider the natural bijection $p_n: 2^{[n]} \longrightarrow 2^{[n]}$ defined by 
\[
p_n(J) := \{ {j+1 \pmod n} \,:\, j \in J\} \qquad (\forall J \subseteq [n]);
\]
the notation $p_n$ being used here for clarity, distinguishing this standard map from our hypothetical $p: \TTT \longrightarrow \TTT$.
The set $2^{[n]}$ is a disjoint union of its $p_n$-orbits, and it suffices to define $p$ as a bijection on $\cDes^{-1}(\OOO)$, for each orbit $\OOO$ separately, so that $\cDes$ becomes equivariant.

Let $\OOO$ be a $p_n$-orbit.
The assumption $m(\varnothing) = m([n]) = 0$ implies that we don't need to consider the (singleton) orbits containing $\varnothing$ and $[n]$.
Since $p_n^n$ is the identity map on $2^{[n]}$, the orbit size $d := \#\OOO$ must be a divisor of $n$; and $d > 1$ since $\varnothing$ and $[n]$ are the only sets fixed by $p_n$. 
Choosing an arbitrary $J \in \OOO$, it follows that $\OOO = \{J, p_n(J), \ldots, p_n^{d-1}(J)\}$. The sizes of the sets $$
\cDes^{-1}(J), \,\, \cDes^{-1}(p_n(J)), \,\,  \cDes^{-1}(p^2_n(J)),\,\,  \ldots, \,\, \cDes^{-1}(p_n^{d-1}(J))
$$ 
are 
$m(J), m(p_n(J)), \ldots, m(p_n^{d-1}(J))$, respectively, and all these numbers are equal due to condition (b).
Define now the bijection 
$$
p: \cDes^{-1}(p_n^i(J)) \longrightarrow \cDes^{-1}(p_n^{i+1}(J))
$$ 
completely arbitrarily, for each $0 \le i \le d-2$. One then sees that this leads to a unique definition of $p: \cDes^{-1}(p_n^{d-1}(J)) \longrightarrow \cDes^{-1}(J)$ which makes $p^d$ the identity map on $\cDes^{-1}(\OOO)$.
The equivariance of $\cDes$ should now be clear.
\end{proof}

\begin{proof}[Proof of Lemma~\ref{t.cyclic_extension}(iii).]
Of course, $m(\varnothing) = \#\cDes^{-1}(\varnothing) =
0$ by property (a).
For each $\varnothing \ne J = \{j_1 < \ldots < j_t\} \subseteq [n]$ and $1 \le i \le t$, one has
$\{j_{i+1} - j_i, j_{i+2}-j_i, \ldots, j_t - j_i\} \subseteq [n-1]$ so that
\[
\begin{aligned}
&\#\Des^{-1}(\{j_{i+1} - j_i, j_{i+2}-j_i, \ldots, j_t - j_i\}) \\
&\quad \overset{(c)}{=} 
   m(\{j_{i+1} - j_i, j_{i+2} - j_i, \ldots, j_t - j_i\}) +
   m(\{j_{i+1} - j_i, j_{i+2}-j_i, \ldots, j_t - j_i, n\}) \\
&\quad \overset{(b)}{=}
   m(\{j_{i+1}, j_{i+2}, \ldots, j_t\}) + 
   m(\{j_i, j_{i+1}, j_{i+2}, \ldots, j_t\}),
\end{aligned}
\]
yielding
\[
\begin{aligned}
&\sum_{i=1}^t (-1)^{i-1} 
  \#\Des^{-1}(\{j_{i+1} - j_i, j_{i+2} - j_i, \ldots, j_t - j_i\}) \\
&= \sum_{i=1}^t (-1)^{i-1} 
  \left( \,\,  m(\{j_{i+1}, j_{i+2}, \ldots, j_t\}) + 
          m(\{j_i, j_{i+1}, j_{i+2}, \ldots, j_t\}) \,\, \right)\\
&= m(J) + (-1)^{t-1} m(\varnothing)
\overset{(a)}{=} m(J). 
\end{aligned}
\]
\end{proof}

\subsection{Univariate generating functions}

The definition of a cyclic extension $(\cDes,p)$ for $\Des$ on $\TTT$ leads immediately to a relation between the ordinary generating functions
\[
\TTT^\cdes(t):=\sum_{T \in \TTT} t^{\cdes(T)}
\]
and
\[
\TTT^\des(t):=\sum_{T \in \TTT} t^{\des(T)}.
\]

\begin{lemma}
\label{lem:cdes-des-distribution-relation}
For any cyclic extension $(\cDes,p)$ of $\Des$ on $\TTT$ 
one has
\[
\frac{n \TTT^{\des}(t)}{(1-t)^{n+1}}
= \frac{d}{dt} \left[ 
\frac{\TTT^{\cdes}(t)}{(1-t)^{n}}\right].
\]
\end{lemma}

\begin{proof}
Since $p$ and each of its powers $p^k$ are bijective, one has 
$$
n \TTT^{\des}(t) = \sum_{T \in \TTT} \sum_{k=0}^{n-1} t^{\des(p^k T)}.
$$
However, equivariance implies that if $\cdes(T)=c$
then the ordered list $(T,pT,p^2T,\ldots,p^{n-1}T)$ contains
\begin{itemize}
\item[$\bullet$]
exactly $c$ entries which have $n$ in their cyclic descent set,
and hence have $\cdes(T)-1$ descents, and
\item[$\bullet$]
the remaining $n-c$ elements have $\cdes(T)$ descents.
\end{itemize}
Therefore the right side above can be rewritten
$$
\sum_{T \in \TTT} \left( \cdes(T) t^{\cdes(T)-1} + (n-\cdes(T)) t^{\cdes(T)} \right)
= n \TTT^\cdes(t) + (1-t) \frac{d}{dt} \TTT^\cdes(t).
$$
It is not hard to check that this is equivalent to the assertion of the lemma.
\end{proof}

\begin{remark}
\label{non-Escher-implies-no-constant-term}
This lemma completely determines $\TTT^{\cdes}(q)$ in terms of $\TTT^{\des}(q)$,
since the non-Escher condition implies that the polynomial $\TTT^{\cdes}(t)$ has no constant term in $t$, and similarly for the formal power series $\TTT^{\cdes}(t)/(1-t)^n$.
\end{remark}

\section{Ribbon Schur functions: affine, cylindric and toric}
\label{sec:Postnikov}

Recall from the introduction that our proof strategy for Theorem~\ref{conj1} involves the introduction of a family of new symmetric functions, which we call {\em affine} (or {\em cyclic}) {\em ribbon Schur functions}.  
We recall here their definition and develop some of their properties, using standard terminology and properties of symmetric functions, as in Macdonald \cite{Macdonald}, Sagan \cite{Sagan}, or Stanley \cite{EC2}.

\subsection{Ribbon Schur functions}
\label{sec:ribbon-schurs}

We start by recalling the classical {\em ribbon Schur functions}.

A {\em (strict) composition}  of $n$ is a sequence $\alpha=(\alpha_1,\ldots,\alpha_t)$ of positive integers satisfying $\alpha_1 + \ldots + \alpha_t = n$.  
Recall from \eqref{composition-of-set} in the Introduction that 
to each subset $J = \{ j_1 < j_2 < \ldots < j_t \} \subseteq [n-1]$ one associates a {\em composition} of $n$, 
\[
\comp{n}{J}:= (j_1,j_2-j_1,j_3-j_2,\ldots,j_t-j_{t-1},n-j_t).
\]

It is customary to associate several symmetric
functions to these objects.  For each composition $\alpha=(\alpha_1,\ldots,\alpha_t)$ 
of $n$, the corresponding {\em homogeneous symmetric function} is
\[
h_{\alpha} := h_{\alpha_1}  h_{\alpha_2} \cdots h_{\alpha_t}
\qquad \text{ where } \qquad
h_k = \sum_{i_1 \le \ldots \le i_k} x_{i_1} \cdots x_{i_k}.
\]
In fact, $h_\alpha$ is a special case of the {\em skew Schur function} $s_{\lambda/\mu}$,
defined generally via $s_{\lambda/\mu} = \sum_{T} \xx^T$,
where $\xx^T:=\prod_{i \in T} x_i$ and $T$ runs through all {\em semistandard (column strict)} tableaux of shape $\lambda/\mu$:
\begin{equation}
\label{every-h-is-a-skew-schur}
h_\alpha = s_{\alpha^\oplus}=s_{(\alpha_1)\oplus \cdots \oplus(\alpha_t)};
\quad \text{e.g.,}\quad
h_{(2,1,3)} = s_{\ytableausetup{boxsize=\mysizetiny}
\begin{ytableau}
\none&\none&\none&   &  & \\
\none&\none&  \\
     &     \\
\end{ytableau}}\,.
\end{equation}
On the other hand, the {\em Jacobi-Trudi formula} expresses any skew Schur
function $s_{\lambda/\mu}$ as a polynomial
in the $h_k$, or linear combination of the $h_\alpha$, as follows:
\[
s_{\lambda/\mu} = \det( h_{\lambda_i-\mu_j+j-i} )_{i,j=1,2,\ldots,\ell(\lambda)},
\]
with the convention that $h_0=1$ while $h_k=0$ for $k < 0$.
In the special case where $\lambda/\mu$ is the ribbon shape having row sizes (from bottom row to to top row) given by the composition $\alpha$ (call this skew shape $\alpha$, by an abuse of notation),
this determinant leads to the formula for $s_\alpha$ given
as \eqref{ribbon-skew-as-alternating-sum} in the Introduction:
if $\alpha=\comp{n}{J}$, then 
\[
s_\alpha = s_{\comp{n}{J}}
= \sum_{\varnothing \subseteq I \subseteq J}
(-1)^{\#(J \setminus I)} h_{\comp{n}{I}}.
\]
For example, if $n=9$ and $J=\{2,6\}$, so that $\alpha=\comp{9}{J}=(2,4,3)$, then $\lambda/\mu = (7,5,2)/(4,1)$ and
\[
\begin{aligned}
s_\alpha = s_{\ytableausetup{boxsize=\mysizetiny}
\begin{ytableau}
\none&\none&\none&\none&     &    &    \\
\none&     &     &     &     \\
     &     \\
\end{ytableau}}
=\det\left[ 
\begin{matrix} 
h_3 & h_{4+3} & h_{2+4+3}\\
1   & h_4 & h_{2+4}\\
0   & 1  & h_2
\end{matrix}
\right]
&= h_{(2,4,3)} - h_{(6,3)}-h_{(2,7)} + h_{(9)} \\
&= h_{\comp{9}{\{2,6\}}} -h_{\comp{9}{\{6\}}} -h_{\comp{9}{\{2\}}} 
   + h_{\comp{9}{\varnothing}}. 
\end{aligned}
\]

A key property of $s_\alpha$ was mentioned already as \eqref{Gessel-ribbon-descent-fact} in the Introduction: 
for any skew shape $\lambda/\mu$, 
the descent map $\Des: \SYT(\lambda/\mu) \longrightarrow 2^{[n-1]}$
has fiber sizes determined by 
\[
\# \Des^{-1}(J) =
\langle s_{\lambda/\mu}, s_{\comp{n}{J}} \rangle
\qquad(\forall J \subseteq [n-1]),
\]
where $\langle -,-\rangle$ is the usual inner product on symmetric functions.

\subsection{Affine ribbon Schur functions}
\label{sec:affine-ribbon-schurs}

We now introduce {\em cyclic} or {\em affine} analogues
of the previous composition and ribbon concepts.  

\begin{definition}
\label{defn_aff}
To each nonempty subset 
$\varnothing \ne J = \{ j_1 < j_2 < \ldots < j_t \} \subseteq [n]$ 
associate (as in \eqref{cyclic-composition-of-set} above) a {\em cyclic composition} of $n$,
\[
\ccomp{n}{J} := (j_2-j_1, \ldots, j_t - j_{t-1}, j_1 + n - j_t).
\]
In particular, $\ccomp{n}{\{j_1\}} := (n)$ while $\ccomp{n}{\varnothing}$ is undefined.
The corresponding {\em affine (or cyclic) ribbon Schur function} is defined (as in \eqref{affine-ribbon-Schur-function} above) by
\[
\cs_{\ccomp{n}{J}} := \sum_{\varnothing \neq I \subseteq J}
(-1)^{\#(J \setminus I)} h_{\ccomp{n}{I}}.
\]
Define also
\[
\cs_{\ccomp{n}{\varnothing}} := 0.
\]
\end{definition}

\begin{example}
We saw that $n=9$ and $J=\{2,6\}$ give rise to the (ordinary) ribbon Schur function
\[
\begin{aligned}
s_{\comp{9}{\{2,6\}}}
&= h_{\comp{9}{\{2,6\}}} - h_{\comp{9}{\{6\}}} - h_{\comp{9}{\{2\}}} 
   + h_{\comp{9}{\varnothing}}\\
&= h_{(2,4,3)} - h_{(6,3)}-h_{(2,7)} + h_{(9)}.
\end{aligned}
\]
By contrast, the corresponding {\em affine} ribbon Schur function is
\[
\begin{aligned}
\cs_{\ccomp{9}{\{2,6\}}}
&= h_{\ccomp{9}{\{2,6\}}} -h_{\ccomp{9}{\{6\}}} -h_{\ccomp{9}{\{2\}}} \\
 &= h_{(4,5)} - h_{(9)}-h_{(9)}
= h_{(4,5)} - 2h_{(9)}.
 \end{aligned}
\]
\end{example}

Our proof strategy for Theorem~\ref{conj1} involves
showing that, whenever $\lambda/\mu$ is a skew shape of size $n$ which is not a connected ribbon, the integers
$m(J):=\langle s_{\lambda/\mu}, \cs_{\ccomp{n}{J}} \rangle$
satisfy conditions (a),(b),(c) of Lemma~\ref{t.cyclic_extension}.
In fact, the nonnegativity condition (a) is the most subtle; 
conditions (b) and (c) follow from the following two propositions.

\begin{proposition}
\label{cyclic-invariance-of-cyclic-compositions}
{\rm (Equivariance)}
For each nonempty subset $\varnothing \ne J \subseteq [n]$, the cyclic composition $\ccomp{n}{p(J)}$ is a cyclic shift of $\ccomp{n}{J}$,
and consequently 
\[
\cs_{\ccomp{n}{p(J)}}=\cs_{\ccomp{n}{J}}.
\]  
\end{proposition}
\begin{proof}
Both assertions follow immediately from the definitions.
\end{proof}

\begin{proposition}\label{t.cor_sum}
{\rm (Extension)}
For each nonempty $\varnothing \ne J \subseteq [n-1]$,
\[
\cs_{\ccomp{n}{J}}+\cs_{\ccomp{n}{J\sqcup\{n\}}}=s_{\comp{n}{J}}.
\]
This also holds for $J = \varnothing$, if $\cs_{\ccomp{n}{\varnothing}}$ is interpreted as 0.
\end{proposition}

\begin{proof}
For $J = \varnothing$ this holds since
\[
\cs_{\ccomp{n}{\{n\}}} 
= s_{\comp{n}{\varnothing}}
= h_{(n)}.
\]
Assume that $J \ne \varnothing$.
By definition,
\[
\cs_{\ccomp{n}{J}} + \cs_{\ccomp{n}{J\sqcup\{n\}}}
= \sum_{\varnothing \neq I \subseteq J} 
 (-1)^{\#(J \setminus I)} h_{\ccomp{n}{I}}
+  \sum_{\varnothing \neq I \subseteq J \sqcup \{n\}}
(-1)^{\#((J \sqcup \{n\}) \setminus I)} h_{\ccomp{n}{I}}.
\]
Each subset $\varnothing \ne I \subseteq J$ contributes 
an $h_{\ccomp{n}{I}}$ to each of the two sums,
but with opposite signs, so they cancel each other.  
The remaining terms (in the second sum)
correspond to subsets $I$ which contain $n$, 
and can be written as 
$I = I' \sqcup \{n\}$ for $I' \subseteq J$:
\[
\cs_{\ccomp{n}{J}} + \cs_{\cc{n}{J\sqcup\{n\}}}
= \sum_{\{n\} \subseteq I \subseteq J \sqcup \{n\}} 
(-1)^{\#((J \sqcup \{n\}) \setminus I)} h_{\ccomp{n}{I}}
= \sum_{I' \subseteq J} 
(-1)^{\#(J \setminus I')} h_{\ccomp{n}{I' \sqcup \{n\}}}.
\]
The cyclic composition $\ccomp{n}{I' \sqcup \{n\}}$ is a (cyclic) rearrangement of the ordinary composition $\comp{n}{I'}$. Thus 
\[
\cs_{\ccomp{n}{J}} + \cs_{\ccomp{n}{J\sqcup\{n\}}}
= \sum_{I' \subseteq J} (-1)^{\#(J \setminus I')} h_{\comp{n}{I'}}
= s_{\comp{n}{J}}. 
\]
\end{proof}

The subtle nonnegativity condition (a) in Lemma~\ref{t.cyclic_extension} will be derived
from the following result, which is the main goal of this section.

\begin{theorem}\label{conj_a3}
For every $J \subseteq [n]$ and non-hook
partition $\lambda\vdash n$, 
\[
\langle \cs_{\ccomp{n}{J}}, s_\lambda \rangle \ge 0.
\]
\end{theorem}

\subsection{Cylindric ribbon shapes and cylindric Schur functions}

The key to Theorem~\ref{conj_a3} is a
relation between the affine ribbon symmetric function
$\cs_{\ccomp{n}{J}}$, defined above, and a special case of 
Postnikov's cylindric Schur functions \cite{Postnikov}
which was introduced implicitly already by 
Gessel and Krattenthaler \cite{GesselKrattenthaler} 
and studied further by McNamara \cite{McNamara}.

In the current subsection we recall the relevant definitions and results regarding cylindric shapes and cylindric Schur functions, in the special case (cylindric ribbons) that we need. These definitions and results are sometimes restated in more convenient terminology, made very explicit when possible.

We start by recalling a special case of Postnikov's {\em cylindric shapes} $\lambda/d/\mu$.

\begin{defn}\label{def:cylindric-ribbon-shape}
(Cf.\ \cite[\S 3]{Postnikov} and \cite[Definition 3.4]{McNamara})
For each subset 
$\varnothing \ne J = \{j_1 < \ldots < j_t\} \subseteq [n]$
define a {\em cylindric ribbon shape} $C_J = \lambda/1/\lambda$ in one of the following equivalent ways:
\begin{itemize}
\item (Postnikov)
The partition
\[
\lambda:= (n - t, j_t - j_1 - (t-1), \ldots, j_3 - j_1 - 2, j_2 - j_1 - 1)
\]
fits inside a $t \times (n-t)$ rectangle, and may be viewed as a lattice path connecting the southwestern and northeastern corners of this rectangle.
Repeat this path, periodically, to obtain an infinite path.
The cylindric ribbon shape $\lambda/1/\lambda$ is the infinite 
ribbon contained between this infinite path and its shift by one step eastward (equivalently, southward).

\item
Consider the cyclic composition 
\[
\ccomp{n}{J}=(j_2-j_1,j_3-j_2,\ldots,j_1+n-j_t).
\]
The cylindric ribbon shape $C_J$ is the infinite ribbon whose sequence of row lengths (from southwest to northeast) is the sequence of parts of this composition, repeated periodically.

\item
Consider the cyclic composition $\ccomp{n}{J}$ above, and let $R_J$ be the corresponding (finite) ribbon shape. Its
northwest boundary is given by the partition $\lambda$ above, except that the first part should be $n-t+1$ rather than $n-t$.
Denote by $a$ (respectively, $b$) the extreme southwestern (respectively, northeastern) square of this ribbon shape.
The cylindric ribbon shape $C_J = \lambda/1/\lambda$ is an infinite 
ribbon made up of copies $(R_i)_{i \in \ZZ}$ of the ribbon $R_J$, placed in the plane such that 
square $a$ of $R_{i+1}$ is immediately north of square $b$ of $R_i$, for all $i \in \ZZ$.
\end{itemize}
\end{defn}

\begin{example}\label{example:cylindric-ribbon-shape}
Let $J=\{1,4,5,8\} \subseteq [9]$, so that $n = 9$ and $t = \#J = 4$.
The corresponding partition and cyclic composition are $\lambda = (5,4,2,2)$ and $\ccomp{n}{J}=(3,1,3,2)$, respectively.
The (finite) ribbon shape $R_J$ is
\[
R_J \, = \,
\ytableausetup{boxsize=\mysize}
\begin{ytableau}
\none & \none & \none & \none &       & b   \\
\none & \none &       &       &       & \none \\
\none & \none &       & \none & \none & \none\\
 a    &       &       & \none & \none & \none
\end{ytableau}
\]
and the cylindric ribbon shape is
\[
C_J = \lambda/1/\lambda = \,
\ytableausetup{boxsize=\mysize}
\begin{ytableau}
 \none & \none & \none & \none & \none & \none & 
 \none & \none & \none & \none & \cdots  \\
 \none & \none & \none & \none & \none & \none & 
 \none & \none & a     &       &   \\
 \none & \none & \none & \none & \none & \none & 
 \none &          & b \\
 \none & \none &\none &\none & \none &           &           &           \\
 \none & \none &\none &\none & \none &           \\
 \none & \none &\none &a       &           &           \\
 \none & \none  &         &b\\
 \cdots &           &          
 \end{ytableau}
\]

%

\end{example}

\begin{remark}\label{remark:extreme_J}
There are two extreme cases deserving special attention.
\begin{itemize}
\item[1.]
In the extreme case $J = [n]$, the partition $\lambda = (0^n)$ corresponds to a vertical path, and the cyclic composition $\ccomp{n}{J}= (1^n)$. The finite ribbon $R_J$ is a column of length $n$, and the cylindric ribbon $C_J = \lambda/1/\lambda$ is an infinite column.
\item[2.]
The other extreme case, $J = \varnothing$, is formally outside the scope of Definition~\ref{def:cylindric-ribbon-shape}. Nevertheless, it is natural to associate with it a row of length $n$ as the finite ribbon $R_J$, and an infinite row as the cylindric ribbon $C_J$.
\end{itemize}
\end{remark}

\medskip


Recall now, from \cite{Postnikov}, the definition of a {\em toric} shape (here -- in the ribbon case).
\begin{definition}\label{def:toric_shape}{\rm \cite[Definition 3.2 and Lemma 3.3]{Postnikov}}
A cylindric ribbon shape $C_J = \lambda/1/\lambda$ is called {\em toric} if (each) one of the following equivalent conditions holds:
\begin{enumerate}
\item
Each row of $C_J$ has length at most $n-t$.
\item
Each column of $C_J$ has length at most $t$.
\item
The first and last columns of $R_J$ have no squares in the same row. 
\end{enumerate}
\end{definition}

\begin{example}
The cylindric ribbon shape in example~\ref{example:cylindric-ribbon-shape} is toric: 
the row lengths of $C_J$ do not exceed $n-t = 5$, 
its column lengths do not exceed $t = 4$, and
the first and last columns of $R_J$ have no squares in the same row.
On the other hand, for the same parmeters $n = 9$ and $t = 4$, the set $J = \{1,2,3,9\} \subseteq [9]$ yields $\lambda = (5,5,0,0)$ and $\ccomp{n}{J} = (1,1,6,1)$, with finite ribbon
\[
R_J \, = \,
\ytableausetup{boxsize=\mysize}
\begin{ytableau}
\none & \none & \none & \none & \none & b   \\
      &       &       &       &       &     \\
      & \none \\
 a    & \none
\end{ytableau}
\]
and cylindric ribbon
\[
C_J = \lambda/1/\lambda = \,
\ytableausetup{boxsize=\mysize}
\begin{ytableau}
 \none & \none & \none & \none & \none & 
 \none & \none & \none & \none & \none & \cdots \\
 \none & \none & \none & \none & \none & 
 \none & \none & \none & \none & \none & a \\
 \none & \none & \none & \none & \none & 
 \none & \none & \none & \none & \none & b \\
 \none & \none & \none & \none & \none & 
       &       &       &       &       & \\
 \none & \none & \none & \none & \none &  \\
 \none & \none & \none & \none & \none & a \\
 \none & \none & \none & \none & \none & b \\
 \cdots &      &       &       &       &   
 \end{ytableau}
\]
This cylindric ribbon shape is {\em not} toric: 
$C_J$ has a row of length $6 > 5 = n-t$
and also a column of length $5 > 4 = t$; and
the first and last columns of $R_J$ have a square in the same row.

\end{example} 

The term {\em toric} reflects a geometric property, as follows.
Consider the cylinder and torus
$$
\begin{aligned}
\Cyl &:= \ZZ^2/(n-t,t)\ZZ,\\
\Tor &:= \ZZ^2/((n-t,0)\ZZ+(0,t)\ZZ),
\end{aligned}
$$
and the natural projection $\pi : \Cyl \to \Tor$. Choosing the coordinate axes properly, a cylindric ribbon shape $C_J \subseteq \ZZ^2$ corresponds to a finite subset of $\Cyl$. The shape $C_J$ is toric if and only if the restriction of $\pi$ to this finite subset is injective, so that, in a sense, $C_J$ can be embedded into the torus.

\begin{lemma}\label{t:toric-conditions}
Let $\varnothing \ne J = \{j_1 < \ldots < j_t\} \subseteq [n]$.
The following conditions are equivalent:
\begin{enumerate}
\item
The cylindric ribbon shape $C_J = \lambda/1/\lambda$ is not toric.
\item
The cyclic composition $\ccomp{n}{J}$ has at most one part of size greater than $1$.
\item
$J$ is a cyclic shift of the set $\{1, 2, \ldots, t\}$.
\item
The ribbon shape $R_J$ is a cyclic shift (in terms of row lengths) of the hook shape $(n-t+1,1^{t-1})$.
\item
$\lambda = ((n-t)^i, 0^{t-i})$ for some $1 \le i \le t$.
\end{enumerate}
\end{lemma}
\begin{proof}
By Definition~\ref{def:toric_shape}, the cylindric shape $C_J$ is not toric if and only if it has a row of length greater than $n-t$.
The row lengths of $C_J$ are the parts of the cyclic composition $\ccomp{n}{J} = (\alpha_1, \ldots, \alpha_t)$.
Since
\[
(\alpha_1 - 1) + \ldots + (\alpha_t - 1) = n - t
\]
with all summands nonnegative, 
a summand $\alpha_i - 1$ can be (at least) $n-t$ if and only if all other summands are zero. It is easy to see that this is equivalent to each of (2)--(5).
\end{proof}


\begin{definition}\label{def:cylindric-Schur}{\rm \cite[\S 5]{Postnikov}}
Let $\lambda/1/\lambda$ be a cylindric ribbon shape. Define the corresponding {\em cylindric Schur function} by
\[
s_{\lambda/1/\lambda}(x_1,\ldots) := \sum_T {\xx}^T,
\]
where summation is over all {\em semistandard cylindric tableaux} $T$ filling the shape $\lambda/1/\lambda$ with entries in $\{1,2,\ldots\}$. This means that the entries of $T$ are ``$(t,n)$-periodic'' (see \cite[Figure 4]{Postnikov}), weakly increasing from left to right in rows and strictly increasing from top to bottom in columns.
Equivalently, such a filling corresponds to a semistandard tableau $T$ of the (finite) ribbon shape $R_J$ with the extra condition
\[T_a < T_b\,,
\]
where the squares $a, b \in R_J$ are as in Definition~\ref{def:cylindric-ribbon-shape}.
\end{definition}

\begin{example}\label{example-cylindric-tableau}
A semistandard cylindric tableau filling the cylindric ribbon shape of Example~\ref{example:cylindric-ribbon-shape} is, e.g.,
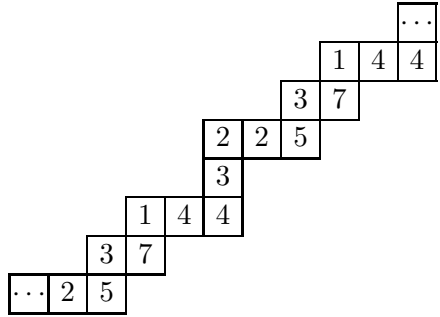
\[
\ytableausetup{boxsize=\mysize}
\begin{ytableau}
 \none & \none & \none & \none & \none & 
 \none & \none & \none & \none & \none & \cdots \\
 \none & \none & \none & \none & \none & 
 \none & \none & \none & 1     & 4     & 4 \\
 \none & \none & \none & \none & \none & 
 \none & \none & 3     & 7 \\
 \none & \none & \none & \none & \none & 
 2 & 2 & 5 \\
 \none & \none & \none & \none & \none & 3 \\
 \none & \none & \none & 1     & 4     & 4 \\
 \none & \none & 3     & 7 \\
 \cdots & 2 & 5         
 \end{ytableau}
\]
The corresponding semistandard tableau $T$ of the (finite) ribbon shape $R_J$ with the extra condition $T_a < T_b$ is
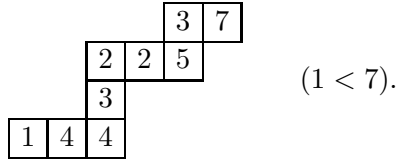
\[
\ytableausetup{boxsize=\mysize}
\begin{ytableau}
\none & \none & \none & \none & 3 & 7 \\
\none & \none & 2 & 2 & 5 \\
\none & \none & 3 \\
 1 & 4 & 4
\end{ytableau}
\qquad (1 < 7).
\]
\end{example}

\subsection{Affine vs.\ cylindric ribbon Schur functions}

Our next result shows that our affine ribbon Schur functions 
(Definition~\ref{sec:affine-ribbon-schurs}) are almost the same as Postnikov's cylindric ribbon Schur functions (Definition \ref{def:cylindric-Schur}).
To state it, recall the {\em power sum} symmetric function defined by
\[
p_n := x_1^n + x_2^n + \ldots
\]
and its well-known  \cite[Theorem 7.17.1 for $\mu = \varnothing$]{EC2} 
expansion into Schur functions
\begin{equation}\label{e:power_sum}
p_n = \sum_{k = 0}^{n-1} (-1)^{k} s_{(n-k,1^k)}.
\end{equation}

\begin{proposition}
\label{toric-zigzag-conjecture}
For $\varnothing \ne J \subseteq [n]$, with associated cylindric ribbon shape $C_J = \lambda/1/\lambda$, one has
\[
s_{\lambda/1/\lambda} = \cs_{\ccomp{n}{J}} + (-1)^{\#J} p_n.
\]
\end{proposition}

\begin{proof}
As explained in \eqref{every-h-is-a-skew-schur}, one has
\begin{equation}
\label{cyclic-h-tableau-sum}
h_{\ccomp{n}{J}} = \sum_T {\xx}^T
\end{equation}
summing over all semistandard tableaux $T$ 
filling the horizontal strip $\ccomp{n}{J}^\oplus$.
For $i=1,2,\ldots,t$ consider the $i^{th}$ row (from the bottom) of this horizontal strip, 
and label its leftmost and rightmost entries by $x_i$ and $y_i$, respectively.
For Example~\ref{example:cylindric-ribbon-shape} with $n=9$, $J=\{1,4,5,8\}$ and $\ccomp{n}{J}=(3,1,3,2)$, the horizontal strip is:
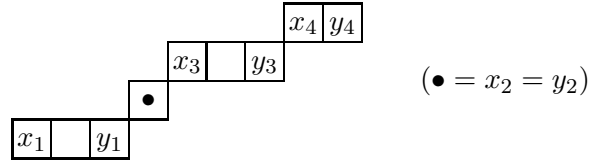
\[
\ytableausetup{boxsize=\mysize}
\begin{ytableau}
 \none & \none & \none & \none & \none & 
 \none & \none & x_4   & y_4   \\
 \none & \none & \none & \none & x_3   & 
       &  y_3  & \none & \none \\
 \none & \none & \none & \bullet \\
 x_1  &       & y_1   & \none & \none & 
 \none& \none & \none & \none
\end{ytableau}
\qquad (\bullet = x_2 = y_2)
\]
Comparing this to the corresponding (finite) ribbon
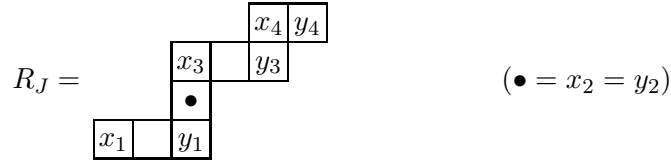
\[
R_J = \,
\ytableausetup{boxsize=\mysize}
\begin{ytableau}
 \none & \none & 
 \none & \none & x_4   & y_4   \\
 \none & \none & x_3   & 
       &  y_3  & \none & \none \\
 \none & \none & \bullet \\
 x_1  &       & y_1   & \none & \none & 
 \none& \none & \none & \none
\end{ytableau}
\qquad (\bullet = x_2 = y_2)
\]
we see that $s_{\lambda/1/\lambda}$ is equal to the same sum as in \eqref{cyclic-h-tableau-sum}, but only over those
$T$ which satisfy the strict inequalities 
\[
x_i < y_{i-1} \qquad (1 \le i \le t),
\]
where subscripts are interpreted modulo $t$ so that the first inequality is $x_1 < y_0 = y_t$.

Define, for each $T$ appearing in \eqref{cyclic-h-tableau-sum},
its {\em violation set} 
\[
V(T) := \{i \,:\, y_{i-1} \le  x_i\} \subseteq [t].
\]
Inclusion-exclusion gives
\[
s_{\lambda/1/\lambda}
= \sum_{T\,:\, V(T) = \varnothing} \xx^T
= \sum_{A \subseteq [t]} (-1)^{\#A} \sum_{T \,:\, V(T) \supseteq A} \xx^T.
\]
A violation at $i$ means that $y_{i-1} \le x_i$, 
so that row $i$ (from the bottom) may be juxtaposed 
at the end of row $i-1$, keeping $T$ semistandard . 
Thus, for each $A \subseteq [t]$ other than $A = [t]$,
the rightmost sum is over all semistandard tableaux filling
the horizontal strip $\ccomp{n}{I}$ for 
$I := \{j_i \,;\, i \in [t] \setminus A\} \ne \varnothing$,
and hence the sum is $h_{\ccomp{n}{I}}$.  
For $A=[t]$ (i.e., $I = \varnothing$),
the sum is over the {\em constant} fillings $T$, 
yielding $x_1^n + x_2^n + \ldots = p_n$.
Therefore
\[
s_{\lambda/1/\lambda}
= (-1)^{\#J} p_n
+ \sum_{\varnothing \neq I \subseteq J} 
(-1)^{\# (J \setminus I)} h_{\ccomp{n}{I}}
= (-1)^{\#J} p_n + \cs_{\ccomp{n}{J}}. 
\]
\end{proof}

For an algebraic consequence of the toric property of a shape, consider the specialization $x_{t+1} = x_{t+2} = \ldots = 0$.
\begin{proposition}\label{t:non-toric-vanishes}{\rm \cite[Lemma 5.2]{Postnikov}}
The cylindric (toric) Schur polynomial $s_{\lambda/1/\lambda}(x_1, \ldots, x_t)$ is nonzero if and only if the shape $\lambda/1/\lambda$ is toric.
\end{proposition}

Together with Lemma~\ref{t:toric-conditions}, this implies
\begin{corollary}\label{t:hook-vanishes}
For $\lambda = ((n-t)^i, 0^{t-i})$  with $1 \le i \le t$,
\[
s_{\lambda/1/\lambda}(x_1, \ldots, x_t) = 0.
\]
\end{corollary}
In fact, it is not difficult to prove Corollary~\ref{t:hook-vanishes} directly from Definition~\ref{def:cylindric-Schur}, since the shape $C_J$ has a column of length $t+1$ which cannot be filled by distinct numbers in $[t]$.

\begin{proposition}\label{t.Gromov-Witten-corollary}{\rm (Cf.\ McNamara \cite[Lemma 5.3 and Proposition 5.5]{McNamara})}
Fix $1 \le t \le n$. Then
\begin{equation}
\label{e.cyclic-ribbon-expansion}
\cs_{\ccomp{n}{[t]}}=\sum_{k=0}^{t-1} (-1)^{t-1-k} s_{(n-k,1^k)}
\end{equation}
and, for each nonempty subset $\varnothing \ne J \subseteq [n]$ with $\#J = t$, there exist nonnegative integers $c_{J,\nu}$ such that
\begin{equation}
\label{e.Gromov-Witten-expansion}
\cs_{\ccomp{n}{J}}
= \cs_{\ccomp{n}{[t]}} 
+ \sum_{\text{non-hook } \nu \vdash n} c_{J,\nu} s_\nu.
\end{equation}
\end{proposition}

\begin{proof}
We claim that it suffices to prove both assertions after specializing them to 
the finite variable set $\{x_1,\ldots,x_t\}$, namely
letting $x_{t+1} = x_{t+2}= \ldots = 0$.  
Indeed, this specialization annihilates each Schur function $s_{\nu}$ with $\ell(\nu) > t$, whereas the surviving $\{s_\nu(x_1,\ldots,x_t)\}_{\ell(\nu) \leq t}$ form a basis for the symmetric polynomials in $x_1,\ldots,x_t$; see
\cite[Chap. I, (3.2)]{Macdonald}.
Thus our claim will follow once we check that, for $t = \#J$,
$\cs_{\ccomp{n}{J}}$ always lies in the linear span of the Schur functions $s_\nu$ with $\ell(\nu) \leq t$.
According to Definition~\ref{defn_aff},
$\cs_{\ccomp{n}{J}}$ is an alternating sum of $h_{\ccomp{n}{I}}$ 
with each $I$ of size $ 1 \le \#I \le t$, so that
$\alpha = \ccomp{n}{I}$ is a composition of $n$ with at most $t$ parts,
and our claim follows from {\em Young's rule} 
\cite[Corollary 7.12.4]{EC2}: 
\[
h_\alpha=\sum_{\nu} K_{\nu,\alpha} s_\nu,
\]
where $K_{\nu,\alpha}$ is the number of semistandard tableaux of shape $\nu$ having $\alpha_j$ occurrences of the entry $j$ (for each $j$), which is zero if $\ell(\nu) > t$.

Now, letting $x_{t+1} = x_{t+2}= \ldots = 0$, equation~\eqref{e:power_sum} specializes to
\[
p_n(x_1, \ldots, x_t) = 
\sum_{k = 0}^{t-1} (-1)^{k} s_{(n-k,1^k)}(x_1, \ldots, x_t),
\]
since $s_{(n-k,1^k)}(x_1,\ldots,x_t)=0$ for $k \ge t$.
We can thus rephrase the two assertions of our proposition as follows:
for each $J \subseteq [n]$ of size $t$ there exist $c_{J,\nu} \ge 0$ such that 
\[
\cs_{\ccomp{n}{J}}(x_1,\ldots,x_t)
= (-1)^{t-1} p_n(x_1,\ldots,x_t) 
+ \sum_{\text{non-hook } \nu} c_{J,\nu} s_\nu(x_1,\ldots,x_t),
\]
and if $J=[t]$ then all $c_{J,\nu}=0$.
Using Proposition~\ref{toric-zigzag-conjecture}, this is equivalent 
to the assertion that there exist $c_{J,\nu} \geq 0$ such that 
\[
s_{\lambda/1/\lambda}(x_1,\ldots,x_t) 
= \sum_{\text{non-hook } \nu} c_{J,\nu}  s_\nu(x_1, \ldots, x_t),
\]
and that for $J=[t]$ they all vanish, 
i.e.,  $s_{\lambda/1/\lambda}(x_1,\ldots,x_t)=0$.
Indeed, for $J = [t]$ one has $\lambda = (n-t, 0^{t-1})$ and therefore, by Corollary~\ref{t:hook-vanishes}, $s_{\lambda/1/\lambda}(x_1,\ldots,x_t)=0$.
On the other hand, for any $J \subseteq [t]$, 
the inequalities $c_{J,\nu} \ge 0$ follow from
Postnikov's result \cite[Theorem 5.3]{Postnikov} that
\[
s_{\lambda/d/\mu}(x_1,\ldots,x_t)
= \sum_{\nu \subseteq [t] \times [n-t]} C^{\lambda,d}_{\mu,\nu} s_\nu(x_1, \ldots, x_t)
\]
where the sum is over all shapes $\nu$ contained in a $t \times (n-t)$ rectangle, and $C^{\lambda,d}_{\mu,\nu}$ are
{\em Gromov-Witten invariants} appearing as structure constants in the quantum cohomology of Grassmannians, and known to be nonnegative. Since no hook shape of size $n$ is contained in a $t \times (n-t)$ rectangle, the sum is over non-hook shapes only, for which $c_{J,\nu} = C^{\lambda,1}_{\lambda,\nu} \ge 0$, and the result follows.
\end{proof}

Theorem~\ref{conj_a3} is clearly a consequence of Proposition~\ref{t.Gromov-Witten-corollary}, since for any nonempty subset $J \subseteq [n]$ and any non-hook partition $\nu \vdash n$
\[
\langle \cs_{\ccomp{n}{J}}, s_\nu \rangle
= \langle s_{\lambda/1/\lambda}, s_\nu \rangle
= C^{\lambda,1}_{\lambda,\nu} \ge 0.
\]

\medskip
It is worth noting that the Gromov-Witten invariants $C^{\lambda,d}_{\mu,\nu}$ have several interpretations, in addition to the one given in the proof of Proposition~\ref{t.Gromov-Witten-corollary}:
\begin{itemize}
\item[$\bullet$] 
They count certain {\em puzzles}, as conjectured by Knutson
and proved by Buch, Kresch, Purbhoo, and Tamvakis~\cite{BKPT}.
\item[$\bullet$] 
They have algebraic interpretations involving Morse's {\em $k$-Schur functions}, and in the {\em Verlinde fusion algebra};
see, e.g., the background discussion by Morse and Schilling \cite[\S 1.4]{MorseSchilling}.
\item[$\bullet$] 
Pawlowski has proved \cite[Theorem 7.8]{Pawlowski} a conjecture of Postnikov \cite[Conjecture 9.1]{Postnikov}, asserting that that $s_{\lambda/d/\mu}$ is the Frobenius characteristic for the Specht module of the toric shape ${\lambda/d/\mu}$, so that $C^{\lambda,d}_{\mu,\nu}$ are its irreducible expansion coefficients.
\end{itemize}

In the special case where $d=1$ and $\lambda=\mu$, the shape $\lambda/1/\lambda = C_J$ corresponds to some $J \subseteq [n]$,
and then, for a (non-hook) shape $\nu$, one can regard Theorem~\ref{conj1} as yielding another interpretation: 
\[
C^{\lambda,1}_{\lambda,\nu} = \langle \cs_{\ccomp{n}{J}}, s_\nu \rangle
= \#\{ T \in \SYT(\nu): \cDes(T)=J \},
\]
where $(\cDes,p)$ is any cyclic extension of $\Des$ on $\SYT(\nu)$.

\subsection{Hook multiplicities in skew Schur functions}
\label{sec:lemmas}

We will now compute, for future use, the multiplicities of hook Schur functions in the Schur expansion of a skew Schur function.
Recall that a {\em ribbon} is a connected skew shape which 
does not contain a $2 \times 2$ square. A {\em generalized ribbon} is a skew shape all of whose connected components are ribbons. 
The {\em height} of a skew shape is the number of its nonempty rows. 

\begin{lemma}\label{lemma_prel2}
Fix a hook shape $(n-k,1^k)$ with $0 \le k \le n-1$.
Then, for a skew shape $\lambda/\mu$ with $n$ cells and $m$ connected components,
\[
\langle s_{\lambda/\mu}, s_{(n-k,1^k)} \rangle =
\begin{cases}
\binom{m-1}{h-k-1}, & \text{if }   \lambda/\mu \text{ is a generalized ribbon of height } k+1 \le h \le k+m;\\
0, & \text{otherwise.}
\end{cases}
\]
\end{lemma}
\begin{proof}
By the Littlewood-Richardson rule ~\cite[Theorems A1.3.3 and A1.3.8]{EC2}), 
the LHS is the number of 
{\em semi-standard} (column-strict) Young tableaux $T$ of shape $\lambda/\mu$,
filled with $n-k$ copies of $1$ and exactly one copy of each of 
$2, 3,\ldots, k+1$, 
such that the {\em reading word} of $T$ (the concatenation of all its rows, bottom to top) has the sequence $k+1,k,\ldots,2,1$ as a subword.

There are no such tableaux $T$ if $\lambda/\mu$ contains a $2 \times 2$ square,
since this square must be filled with 
\[
\ytableausetup{boxsize=\mysizesmall}
\begin{ytableau}
a & b  \\
c & d 
\end{ytableau}
\]
where $a < c \le d$, implying $1 < c < d$ and violating the subword condition.

Thus $\lambda/\mu$ must be a generalized ribbon, say with $m$ connected components and height $h$.
The subword condition implies that each row of $T$ can have at most one entry which is not $1$, and that these entries are $k+1, k, \ldots, 2$, from bottom to top. The semi-standard property of $T$ implies that each such entry is in the easternmost square in its row, and the ribbon shapes of the components imply that the easternmost entry in a row can be a $1$ only for the top row of a component.

Summing up, let $I$ be the set of all connected components that have a $1$ at their northeastern corner. Then $I$ completely determines the tableau $T$: all the top rows corresponding to $I$ end with a $1$, all the other rows in $T$ end with a non-$1$, namely with $2, \ldots, k, k+1$ in increasing order from top to bottom, and all other entries are $1$. The set $I$ must include the top (northeastern) component, by the subword property, but is otherwise free. The number of possible such $I$ (i.e., the number of tableaux $T$) is $\binom{m-1}{\#I-1}$.
The observations that $1 \le \#I \le m$ and $h = k + \#I$ (by row counting) complete the proof.
\end{proof}

\begin{corollary}\label{lemma_prel4}
For a generalized ribbon $\lambda/\mu$ of size $n$ and height $h$, with $m \geq 2$ components, there exist nonnegative coefficients $c_\nu \geq 0$ such that
\[
\begin{aligned}
s_{\lambda/\mu} 
&= \sum_{\text{non-hook }\nu\vdash n} c_\nu s_\nu 
+ \sum_{k=h-m}^{h-1} \binom{m-1}{h-1-k} s_{(n-k, 1^k)} \\
&= \sum_{\text{non-hook }\nu\vdash n} c_\nu s_\nu 
+ \sum_{k=h-m+1}^{h-1} \binom{m-2}{h-1-k} s_{(1^k) \oplus (n-k)}
\end{aligned}
\]
\end{corollary}

\begin{proof}
The first equality follows from Lemma~\ref{lemma_prel2}, since the skew Schur function $s_{\lambda/\nu}$ is a nonnegative linear combination of Schur functions.
The second equality follows, then, from the equation
\begin{equation}\label{e.sum_of_hooks}
s_{(1^k) \oplus (n-k)}=s_{(n-k+1,1^{k-1})}+s_{(n-k,1^k)}
\qquad (1 \le k \le n-1).
\end{equation}
This equation, in turn, follows easily from the definition of a Schur function as a sum over semi-standard tableaux $T$, distinguishing the cases $T_a \ge T_b$ and $T_a < T_b$, where $a$ is the first (westernmost) square of the row shape $(n-k)$ and $b$ is the top square of the column shape $(1^k)$.  
\end{proof}

An immediate consequence of Proposition~\ref{t.Gromov-Witten-corollary} is the following.

\begin{corollary}
\label{affine-ribbon-hook-pairing} 
For $\varnothing \ne J \subseteq [n]$ with $t:=\#J$ and $0 \le k \le n-1$,
\[
\langle \cs_{\ccomp{n}{J}}, s_{(n-k,1^k)} \rangle =\begin{cases}
(-1)^{t-1-k} & \text{ if } 0 \le k \le t-1;\\
0 & \text{ otherwise. }
\end{cases}
\]
\end{corollary}
%
%

Together with equation \eqref{e.sum_of_hooks}, this implies the following.
\begin{corollary}
\label{cor:crucial-hook-fact}
For $\varnothing \ne J \subseteq [n]$ with $t=\#J$ and $1 \le k \le n$,
\[
\langle \cs_{\ccomp{n}{J}}, s_{(1^k) \oplus (n-k)}  \rangle =
\begin{cases}
1 & \text{ if } k=t,\\
0 & \text{ otherwise.}
\end{cases}
\]
\end{corollary}

\section{Proof of Theorem~\ref{conj1}}
\label{sec:conj1-proof}

Recall the statement of the theorem.

\vskip.1in
\noindent
{\bf Theorem~\ref{conj1}.}
{\em
Let $\lambda/\mu$ be a skew shape.
The descent map $\Des$ on $\SYT(\lambda/\mu)$ has a cyclic extension $(\cDes,p)$ if and only if $\lambda/\mu$ is not a connected ribbon.
Furthermore, for all $J \subseteq [n]$, all such cyclic extensions share the same cardinalities $\#\cDes^{-1}(J)$.
}

\begin{proof}
For the ``if'' direction, fix a skew shape $\lambda/\mu$ which is not a connected ribbon, and let $\TTT := \SYT(\lambda/\mu)$ with the
usual descent map $\Des: \TTT \longrightarrow 2^{[n-1]}$.
We will use Lemma~\ref{t.cyclic_extension}(ii) to show that a cyclic extension $(\cDes,p)$ exists. Indeed, define
\[
m(J) := \langle  s_{\lambda/\mu}, \cs_{\ccomp{n}{J}} \rangle
\qquad (\forall J \subseteq [n])
\]
where, by definition, $\cs_{\ccomp{n}{\varnothing}} := 0$.

Conditions (b) and (c) of Lemma~\ref{t.cyclic_extension} follow from Propositions \ref{cyclic-invariance-of-cyclic-compositions} and \ref{t.cor_sum}, respectively.
It remains to check condition (a):
\[
m(J) \ge 0 \text{ for all } J, \text{ with } m(\varnothing) = m([n]) = 0.
\]
We start with the extreme cases $J = \varnothing$ and $J = [n]$.
Indeed, $m(\varnothing) = 0$ since $\cs_{\ccomp{n}{\varnothing}} = 0$ by definition. Note that, in fact, Proposition~\ref{toric-zigzag-conjecture} holds even in this case if the cylindric ribbon shape $C_J = \lambda/1/\lambda$ is interpreted as an infinite row (see Remark~\ref{remark:extreme_J}(2)), for which $s_{\lambda/1/\lambda} = \sum_i x_i^n = p_n$ by definition.
For $J = [n]$, the cylindric ribbon shape $C_J = \lambda/1/\lambda$ is an infinite column (see Remark~\ref{remark:extreme_J}(1)). There are no periodic semi-standard tableaux of this shape, so that $s_{\lambda/1/\lambda} =0$. Thus, by Proposition~\ref{toric-zigzag-conjecture} and equation \eqref{e:power_sum},
\[
\cs_{\ccomp{n}{[n]}} = (-1)^{n-1} p_n
= \sum_{k=0}^{n-1} (-1)^{n-1-k} s_{(n-k,1^k)}.
\]
By Lemma~\ref{lemma_prel2}, if $\lambda/\mu$ is a generalized ribbon of height $h$ with $m \ge 2$ connected components, then
\[
\begin{aligned}
m([n]) 
&= \langle  s_{\lambda/\mu}, \cs_{\ccomp{n}{[n]}} \rangle
= \sum_{k=0}^{n-1} (-1)^{n-1-k} \langle  s_{\lambda/\mu}, s_{(n-k,1^k)} \rangle \\
&=\sum_{k=h-m}^{h-1} (-1)^{n-1-k} \binom{m-1}{h-k-1}
= \sum_{j=0}^{m-1} (-1)^{n-h+j} \binom{m-1}{j}
= 0.
\end{aligned}
\]
If $\lambda/\nu$ is not a generalized ribbon then all the summands are obviously $0$, by Lemma~\ref{lemma_prel2}, whereas generalized ribbons with $m = 1$ are connected ribbons and are explicitly excluded by assumption.

\medskip
We turn now to proving that $m(J) \ge 0$ for all $J \subseteq [n]$.
By the foregoing we may assume, of course, that $J \ne \varnothing$.
Expanding
\[
s_{\lambda/\mu} = \sum_\nu c_\nu s_{\nu}
\]
with nonnegative integer coefficients $c_\nu$, one has 
\[
m(J) = \langle  s_{\lambda/\mu}, \cs_{\ccomp{n}{J}} \rangle
=  \sum_\nu c_\nu \langle s_{\nu}, \cs_{\ccomp{n}{J}} \rangle.
\]

Assume first that $\lambda/\mu$ {\em is not} a generalized ribbon. It suffices to show that
\[
c_\nu \ne 0 \Longrightarrow \langle s_\nu, \cs_{\ccomp{n}{J}} \rangle \ge 0.
\]
For non-hook shapes $\nu$ this holds by Theorem~\ref{conj_a3}, and for hook shapes $\nu$ this holds (as $c_\nu = 0$) by Lemma~\ref{lemma_prel2}.

Finally, assume that $\lambda/\mu$ {\em is} a generalized ribbon of height $h$ with $m \ge 2$ components. 
Again, $\langle s_\nu, \cs_{\ccomp{n}{J}} \rangle \ge 0$ for non-hook shapes $\nu$ by Theorem~\ref{conj_a3}. For hook shapes this doesn't always hold for the {\em individual} inner products, but the sum
\[
\begin{aligned}
\sum_{\text{hook } \nu \vdash n} c_\nu \langle s_{\nu}, \cs_{\ccomp{n}{J}} \rangle
&= \sum_{k=0}^{n-1} c_{(n-k,1^k)} \langle s_{(n-k,1^k)}, \cs_{\ccomp{n}{J}} \rangle
= \sum_{k=0}^{\#J-1} c_{(n-k,1^k)} (-1)^{\#J-1-k} \\
&= \sum_{k=h-m}^{\min(h-1,\#J-1)} \binom{m-1}{h-k-1} (-1)^{\#J-1-k},
\end{aligned}
\]
by Corollary~\ref{affine-ribbon-hook-pairing} and Lemma~\ref{lemma_prel2}.
If $h \le \#J$ then this sum is
\[
\sum_{k=h-m}^{h-1} \binom{m-1}{h-k-1} (-1)^{\#J-1-k} = 
\sum_{j=0}^{m-1} \binom{m-1}{j} (-1)^{\#J-h+j} = 0
\]
since $m \ge 2$;
and otherwise (namely, if $h > \#J$) it is
\[
\begin{aligned}
\sum_{k=h-m}^{\#J-1} \binom{m-1}{h-k-1} (-1)^{\#J-1-k}
&= \sum_{j=h-\#J}^{m-1} \binom{m-1}{j} (-1)^{\#J-h+j} \\
&= \sum_{j=h-\#J}^{m-1} \left[ \binom{m-2}{j-1} + \binom{m-2}{j} \right] (-1)^{\#J-h+j} \\
&= \binom{m-2}{h-\#J-1} \ge 0.
\end{aligned}
\]

%

\medskip
For the ``only if'' direction, let $\lambda/\mu$ be a connected ribbon.
We must show that the descent map on $\SYT(\lambda/\mu)$ {\em does not} have a cyclic extension $(\cDes,p)$. Assume the contrary.

If $\lambda/\mu$ is a single row $(n)$ or a single column $(1^n)$,
the set $\TTT=\SYT(\lambda/\mu)$ contains a unique tableau $T$ which has $\Des(T) \in \{\varnothing, [n-1]\}$. The extension and equivariance properties of $(\cDes,p)$ force $\cDes(T) \in \{\varnothing, [n]\}$, contradicting the non-Escher condition.

Assume now that $\lambda/\mu \ne (n), (1^n)$ is a connected ribbon of height (i.e., number of rows) $h$; by assumption, $2 \le h \le n-1$.
Then there exists a standard tableau $T_0 \in \SYT(\lambda/\mu)$ with
$\Des(T_0)=\{1,2,\ldots, h-1\}$, built as in the following example (for $n=14$ and $h=6$):

\ytableausetup{boxsize=\mysize}
$$
\begin{ytableau}
\none &\none &\none      &\none &\none &\none      &\none &\mathbf{1}
&14 \\
\none &\none &\none      &\none &\none &\mathbf{2} & 12  &13 \\
\none &\none &\none      &\none &\none &\mathbf{3} \\
\none &\none &\none      &\none &\none &\mathbf{4} \\
\none &\none &\mathbf{5} &9     &10    &11         \\
6     &7     &8
\end{ytableau}
$$
By the extension property, $\cDes(T_0)$ is either $\{1,2,\ldots,h-1,n\}$ or $\{1,2,\ldots,h-1\}$.
In the former case, equivariance implies that $\cDes(p(T_0)) = \{1,2,\ldots,h\}$ and therefore also $\Des(p(T_0)) = \{1,2,\ldots,h\}$.
This is impossible for a tableau of height $h$, since $1, \ldots, h, h+1$ must then appear in distinct rows.
In the latter case, equivariance implies that $\cDes(p^{-1}(T_0)) = \{n,1,2,\ldots,h-2\}$ and therefore $\Des(p^{-1}(T_0)) = \{1,2,\ldots,h-2\}$.
This is impossible for a ribbon tableau with $h$ rows (and, consequently, $n-h+1$ columns), since $h-1, h, \ldots, n$ must then appear in distinct columns.
\end{proof}

Let us state explicitly a consequence of the proof.

\begin{corollary}\label{cor:fiber-sizes-as-inner-product}
Let $\lambda/\mu$ be a skew shape of size $n$ which is not a connected ribbon. For any $J \subseteq [n]$ and every cyclic extension $\cDes$ of the usual descent map on $\SYT(\lambda/\mu)$, the fiber size
\[
\# \cDes^{-1}(J) =
\langle s_{\lambda/\mu}, \cs_{\ccomp{n}{J}} \rangle .
\]
\end{corollary}

For $J=\{j_1,\cdots,j_t\}\subseteq [n]$ let $-J:=\{n-j_1,\dots,n-j_t\}$, where zero is identified with $n$.
One deduces that 

\begin{corollary}\label{cor:fiber-sizes-involution}
	Let $\lambda/\mu$ be a skew shape of size $n$ which is not a connected ribbon. For any $J \subseteq [n]$ and every cyclic extension $\cDes$ of the usual descent map on $\SYT(\lambda/\mu)$, 
	the fiber size
	\[
	\# \cDes^{-1}(J) =\# \cDes^{-1}(-J). 
	\]
\end{corollary}

\begin{proof}
By definition, for every $\varnothing\neq I\subseteq [n]$, 
$\alpha^\cyc(I,n)$ 
and $\alpha^\cyc(-I,n)$ 
have the same parts but in reverse (cyclic) order, hence $h_{\alpha^\cyc(I,n)}=h_{\alpha^\cyc(-I,n)}$. Thus, for every non-ribbon skew shape $\lambda/\mu$ and $\varnothing\neq J\subseteq [n]$,
\[
\begin{aligned}
\# \cDes^{-1}(J)& =
\langle s_{\lambda/\mu}, \cs_{\ccomp{n}{J}} \rangle  \\
&= \left\langle s_{\lambda/\mu},  \sum_{\varnothing \neq I \subseteq J} 
(-1)^{\#(J \setminus I)} h_{\ccomp{n}{I}}\right\rangle  \\
&= \left\langle s_{\lambda/\mu},  \sum_{\varnothing \neq -I \subseteq -J}
(-1)^{\#(-J \setminus -I)} h_{\ccomp{n}{-I}}\right\rangle 
=
\langle s_{\lambda/\mu}, \cs_{\ccomp{n}{-J}} \rangle=\# \cDes^{-1}(-J).
\end{aligned}
\]
\end{proof}

\section{Proof of Theorem~\ref{conj2}}
\label{sec:conj2-proof}

Recall the statement of the theorem.
\vskip.1in
\noindent
{\bf Theorem~\ref{conj2}.}
{\em
Let $\lambda/\mu$ be a skew shape which is not a connected ribbon.
Then any cyclic extension $(\cDes,p)$ of the descent map $\Des$ on $\SYT(\lambda)$ satisfies
\[
\sum_{w \in \symm_n} \ttt^{\cDes(w)}
= \sum_{\substack{\text{non-hook}\\ \lambda \vdash n}}
f^\lambda \sum_{T \in \SYT(\lambda)} \ttt^{\cDes(T)}
\quad + \quad \sum_{k=1}^{n-1} \binom{n-2}{k-1} 
 \sum_{T \in \SYT((1^k)\oplus (n-k))} \ttt^{\cDes(T)}
\]
}

\noindent
This is actually the special case $\alpha=(1,1,\ldots,1)=(1^n)$ of the following more general statement, which requires a bit more notation.
Recall that, for a partition $\lambda \vdash n$ and a composition $\alpha=(\alpha_1,\ldots,\alpha_m)$ of $n$, the {\it Kostka number} 
\[
K_{\lambda, \alpha} :=\langle s_\lambda, h_\alpha\rangle
\] 
counts the column-strict tableaux $T$ of shape $\lambda$ having content $\alpha$, namely $\alpha_i = \#T^{-1}(i)$ ($\forall i$).
In particular, 
\[
K_{\lambda,(1^n)}=f^\lambda=\#SYT(\lambda).
\]
Denote, for a tableau $T$ as above, $\xx^T := x_1^{\alpha_1} \cdots x_m^{\alpha_m}$.

\begin{theorem}\label{prop25}
For every composition $\alpha=(\alpha_1,\ldots,\alpha_m)$  of $n$ with $m \geq 2$,
\[
\sum_{T\in \SYT( \alpha^\oplus)} \ttt^{\cDes(T)}
= \sum_{\substack{\text{non-hook}\\ \lambda \vdash n}}
  K_{\lambda,\alpha}  \sum\limits_{T\in \SYT(\lambda)}
   {\ttt}^{\cDes(T)}
\quad + \quad
  \sum_{k=1}^{m-1}\binom{m-2}{k-1} 
      \sum_{T\in \SYT((1^k)\oplus (n-k))} \ttt^{\cDes(T)},
\]
\end{theorem}

\begin{proof}[Proof of Theorem~\ref{prop25}.]
The horizontal strip $\alpha^\oplus$ is a generalized ribbon with $m \ge 2$ components and height $h = t$. By Corollary~\ref{lemma_prel4},
\[
s_{\alpha^\oplus} 
= \sum_{\text{non-hook } \nu \vdash n} c_\nu s_\nu
+ \sum_{k=1}^{m-1} \binom{m-2}{m-1-k} s_{(1^k) \oplus (n-k)}
\]
for some nonnegative coefficients $c_\nu$.
Since $s_{\alpha^\oplus} = h_\alpha$,
\[
c_\nu = \langle s_{\alpha^\oplus}, s_\nu \rangle 
= \langle h_\alpha, s_\nu \rangle = K_{\nu,\alpha}
\]
and therefore
\begin{equation}\label{ribbon-schur-Kostka-expansion}
s_{\alpha^\oplus} 
= \sum_{\text{non-hook } \nu \vdash n} K_{\nu,\alpha} s_\nu
+ \sum_{k=1}^{m-1} \binom{m-2}{k-1} s_{(1^k) \oplus (n-k)}.
\end{equation}
Consider, for each $J \subseteq [n]$, the coefficient of $\ttt^J$ in the LHS of the statement of the theorem.  By \eqref{cDes-fiber-sizes-as-inner-product} (see the proof of Theorem~\ref{conj1}) it is equal to
\[
\#\{T\in \SYT(\alpha^\oplus):\ \cDes(T)=J\}
= \langle s_{\alpha^\oplus}, \cs_{\ccomp{n}{J}} \rangle,
\]
and by \eqref{ribbon-schur-Kostka-expansion} this is equal to
\[
\begin{aligned}
\langle s_{\alpha^\oplus}, \cs_{\ccomp{n}{J}} \rangle
&= \sum_{\text{non-hook } \nu \vdash n} K_{\nu,\alpha} \langle s_\nu, \cs_{\ccomp{n}{J}} \rangle 
+ \sum_{k=1}^{m-1} \binom{m-2}{k-1}
  \langle s_{(1^k) \oplus (n-k)}, \cs_{\ccomp{n}{J}} \rangle \\
&= \sum_{\text{non-hook } \nu \vdash n} K_{\nu,\alpha} \cdot \#\{T\in \SYT(\nu) \,:\, \cDes(T)=J\} \\
&+ \sum_{k=1}^{m-1} \binom{m-2}{k-1} \cdot \#\{T\in \SYT((1^k) \oplus (n-k)) \,:\, \cDes(T)=J\},
\end{aligned}
\]
which is exactly the coefficient of $\ttt^J$ in the RHS of the statement of the theorem.
\end{proof}

\section{Cyclic Eulerian distributions}
\label{sec:symmetric-group}

\subsection{Cyclic descent generating functions}

The {\em descent number} is the size of the descent set.
For any skew shape $\lambda/\mu$ of size $n$ there is a known expression \cite[equation (7.96)]{EC2} for the generating function of the descent number, $\des$, on standard Young tableaux of shape $\lambda/\mu$:
\begin{equation}
\label{skew-tableau-des-distribution}
\sum_{T \in \SYT(\lambda/\mu)} t^{\des(T)}
= (1-t)^{n+1} \sum_{m \ge 0} s_{\lambda/\mu}(1^{m+1}) t^m.
\end{equation}
Here $s_{\lambda/\mu}(1^m)$ is the specialization of the skew Schur function $s_{\lambda/\mu}(x_1, x_2, \ldots)$ under $x_1 = \ldots = x_m = 1$ and $x_{m+1} = \ldots = 0$. Note that when $\mu = \varnothing$ this becomes even more explicit, through the {\em hook-content formula} \cite[Cor. 7.21.4]{EC2} for the specialization $s_\lambda(1^m)$.
In particular, for the skew shape $(1)^{\oplus n}$ this gives the well-known {\it Carlitz formula} for the {\em Eulerian distribution} on $\symm_n$:
\begin{equation}
\label{Carlitz-formula}
\symm_n^{\des}(t) := \sum_{w \in \symm_n} t^{\des(w)}
= (1-t)^{n+1} \sum_{m \ge 0} (m+1)^n t^m
\end{equation}

An analogous expression for the {\em cyclic descent number} $\cdes$ 
is a corollary of Lemma~\ref{lem:cdes-des-distribution-relation}.

\begin{corollary}\label{cor:cdes_gf_skew}
For any skew shape $\lambda/\mu$ of size $n$ which is not a connected ribbon,
\begin{equation}
\label{skew-tableau-cdes-distribution}
\sum_{T \in \SYT(\lambda/\mu)} t^{\cdes(T)}
= n (1-t)^{n} \sum_{m \ge 1} s_{\lambda/\mu}(1^m) \frac{t^m}{m}.
\end{equation}
In particular, for the skew shape $(1)^{\oplus n}$ this gives
\begin{equation}
\label{cyclic-Carlitz-formula}
\symm_n^{\cdes}(t) 
:= \sum_{w \in \symm_n} t^{\cdes(w)}
= n (1-t)^n \sum_{m \geq 1} m^{n-1} t^m  
= nt\, \symm_{n-1}^\des(t) \qquad (n \ge 2).
\end{equation}
\end{corollary}
\begin{proof}
For  $\TTT=\SYT(\lambda/\mu)$, 
Lemma~\ref{lem:cdes-des-distribution-relation} and Remark~\ref{non-Escher-implies-no-constant-term}
show that $\TTT^\cdes(t)$ is determined by
\[
\frac{d}{dt}
\left[
\frac{\TTT^\cdes(t) }{(1-t)^n}
\right]
=
\frac{n \TTT^\des(t) }{(1-t)^{n+1}}
= n \sum_{m \ge 0} s_{\lambda/\mu}(1^{m+1}) t^m.
\]
This implies \eqref{skew-tableau-cdes-distribution}, which for $\lambda/\mu=(1)^{\oplus n}$ specializes to \eqref{cyclic-Carlitz-formula}.
\end{proof}

We now focus on $\lambda/\mu=(1)^{\oplus n}$, where we can take $\TTT=\symm_n$ and use the extra symmetry to get more refined results.
Consider the {\em multivariate} generating functions
\[
\symm_n^\Des(\ttt)
= \symm_n^{\Des}(t_1,\ldots,t_{n-1})
:= \sum_{w \in \symm_n} \ttt^{\Des(w)}
\]
and
\[
\symm_n^\cDes(\ttt)
= \symm_n^{\cDes}(t_1,\ldots,t_{n-1},t_n)
:= \sum_{w \in \symm_n} \ttt^{\cDes(w)}.
\]
Note that $\symm_n^\Des(\ttt)$ and  $\symm_n^\cDes(\ttt)$ are, respectively, the {\em flag $h$-polynomials} for the type $A_{n-1}$ Coxeter complex and the reduced Steinberg torus considered by Dilks, Petersen, and Stembridge \cite{DPS}; see also Section~\ref{Steinberg-torus-section} below.
The two are related by an obvious specialization
\begin{equation}
\label{the-obvious-multivariate-specialization}
\left[ \symm_n^\cDes(\ttt) \right]_{t_n=1} 
 = \symm_n^\Des(\ttt).
\end{equation}
On the other hand, $\symm_n^\cDes(\ttt)$ and $\symm_{n-1}^\Des(\ttt)$ are also related in a slightly less obvious way. 
Define an action of the cyclic group $\ZZ/n\ZZ=\langle c \rangle=\{e,c,c^2,\cdots,c^{n-1}\}$ 
on $\ZZ[t_1,\ldots,t_n]$ by shifting subscripts modulo $n$, i.e. $c(t_i)=t_{i+1 \pmod n}$.

\begin{proposition}
\label{multivariate-second-approach-proposition}
For $n \geq 2$, one has
\begin{equation}
\label{multivariate-less-trivial-cdes-des-relation-1}
\symm_n^\cDes(\ttt) 
= \sum_{i=1}^n c^i\left( t_n\, \symm_{n-1}^\Des(\ttt) \right)
\end{equation}
and also
\begin{equation}
\label{multivariate-less-trivial-cdes-des-relation-2}
\symm_n^\cDes(\ttt) = g(\ttt) \,\, + \,\, \ttt^{[n]} g(\ttt^{-1}),
\end{equation}
where
\begin{equation}
\label{e.def-of-g}
g(\ttt) = g(t_1,\ldots,t_{n-1})
:= \left[ \symm_n^\cDes(\ttt) \right]_{t_n=0}
= \sum_{i=1}^{n-1} t_i \cdot \left[ c^i\, \symm_{n-1}^\Des(\ttt) \right]_{t_n=0}.
\end{equation}
\end{proposition}

\begin{proof}
Consider the bijection $\symm_n \rightarrow \symm_{n-1} \times [n]$
sending a permutation $w \in \symm_n$ with $w(i)=n$ 
to the pair $(v,i) \in \symm_{n-1} \times [n]$, 
where $v := (w(i+1), \ldots, w(n), w(1), \ldots, w(i-1))$.
The observation that $\ttt^{\cDes(w)} =c^i \left( t_n\, \ttt^{\Des(v)} \right)$ proves \eqref{multivariate-less-trivial-cdes-des-relation-1}.

Define now $g(\ttt)$ by \eqref{e.def-of-g}; 
the last equality there follows from \eqref{multivariate-less-trivial-cdes-des-relation-1}.
The involution $\symm_n \rightarrow \symm_n$ sending
$w=(w_1,\ldots,w_n)$ to $w_0w:=(n+1-w_1,\ldots,n+1-w_n)$
has the property that $\cDes(w_0w)=[n] \setminus \cDes(w)$,
and thus gives a bijection between the permutations $w \in \symm_n$ with $n \not\in \cDes(w)$ and those with $n \in \cDes(w)$.
Since $\ttt^{\cDes(w_0w)} = \ttt^{[n]} (\ttt^{-1})^{\cDes(w)}$,
\[
\symm_n^\cDes(\ttt) 
= \sum_{\substack{w \in \symm_n:\\ n \not\in \cDes(w)}} \ttt^{\cDes(w)} 
+ \sum_{\substack{w \in \symm_n:\\ n \in \cDes(w)}} \ttt^{\cDes(w)} 
= \sum_{\substack{w \in \symm_n:\\ n \not\in \cDes(w)}} 
\left[ \ttt^{\cDes(w)} + \ttt^{[n]} (\ttt^{-1})^{\cDes(w)} \right]
\]
and this proves \eqref{multivariate-less-trivial-cdes-des-relation-2}.
\end{proof}


\begin{remark}
Formulas \eqref{the-obvious-multivariate-specialization} and \eqref{multivariate-less-trivial-cdes-des-relation-1} imply
the following interesting (and seemingly new) recurrence for the ordinary multivariate Eulerian distribution $\symm_n^\Des(\ttt)$:
\[
\symm_n^\Des(\ttt)=
\left[ 
  \sum_{i=1}^n t_i \cdot c^i \symm_{n-1}^\Des(\ttt)
\right]_{t_n=1}.
\]
\end{remark}

\medskip
One can specialize $\symm_n^{\cDes}(\ttt)$ to a {\em bivariate} generating function
\[
\symm_n^\cdes(t,u):=
 \sum_{w \in \symm_n} t^{\des(w)} u^{\cdes(w) - \des(w)}
\]
by setting $t_1=t_2=\cdots=t_{n-1}:=t$ and $t_n := u$.
The following result generalizes an observation of Fulman~\cite{Fulman} and Petersen \cite{Petersen_paper}.

\begin{proposition}
\label{second-approach-proposition}
For $n \geq 2$ one has
\begin{equation}
\label{first-cdes-gf-eqn}
\symm_n^\cdes(t,u) = t^{n-1} f(t^{-1}) + uf(t)
\quad
\text{ where }
\quad
f(t) := \frac{d}{dt} t \symm_{n-1}^\des(t)
\end{equation}
or, equivalently,
\begin{equation}
\label{second-cdes-gf-eqn}
\symm_n^{\cdes}(t,u)
= \left( nt + (u - t) \frac{d}{dt} t \right) \symm_{n-1}^\des(t).           
\end{equation}
\end{proposition}
\begin{proof}
Specializing \eqref{multivariate-less-trivial-cdes-des-relation-2} gives
\[
\symm_n^\cDes(t,u) 
= [\symm_n^\cDes(\ttt)]_{t_1=\cdots=t_{n-1}=t,\, t_n=u}
= g(t) \,\, + \,\, u t^{n-1} g(t^{-1}),
\]
where $g(t)$ is obtained by specializing \eqref{e.def-of-g}:
\[
\begin{aligned}
g(t) 
&= [g(\ttt)]_{t_1=\cdots=t_{n-1}=t} 
= \left[ \sum_{i=1}^{n-1} t_i c^i \symm_{n-1}^{\Des}(\ttt) \right]_{t_1=\cdots=t_{n-1}=t,\, t_n=0} \\
&= \sum_{v \in \symm_{n-1}} \sum_{i=1}^{n-1} 
\left[ t_i c^i \ttt^{\Des(v)} \right]_{t_1=\cdots=t_{n-1}=t,\, t_n=0} 
=
\sum_{v \in \symm_{n-1}} (n-1-\des(v)) \cdot t^{\des(v)+1}.
\end{aligned}
\]
The last equality follows from the fact that, when $v \in \symm_{n-1}$ has $k$ descents, it has $n-1-k$ ascents, and hence exactly $n-1-k$ of the monomials $c^i \left(\ttt^{\Des(v)}\right)$ survive upon setting $t_n=0$. Thus
\[
g(t)=
\sum_{v \in \symm_{n-1}} (n-1-\des(v)) \cdot t^{\des(v)+1}
=t\left( n-1-t\frac{d}{dt}\right) \symm_{n-1}^\des(t)
=t\left( n-\frac{d}{dt} t\right) \symm_{n-1}^\des(t)
\]
and
\[
f(t) 
:= t^{n-1} g(t^{-1}) 
= \sum_{v \in \symm_{n-1}} (n-1-\des(v)) \cdot t^{n-2-\des(v)}
= \sum_{w \in \symm_{n-1}} (\des(w)+1) \cdot t^{\des(w)}
= \frac{d}{dt} t \symm_{n-1}^\des(t),
\]
completing the proof. 
\end{proof}

\begin{remark}
The coefficients of $f(t)=\frac{d}{dt} t \symm_{n-1}^\des(t)$
appear as OEIS entry {\tt A065826}.
\end{remark}

The preceding calculations lead to an exponential
generating function for $\symm_n^\cdes(u,t)$,
generalizing work of Petersen \cite[Proposition 14.4]{Petersen_book}.
Recall that the Eulerian distribution on $\symm_n$ 
\[
\symm_n^{\des}(t):=\sum_{w \in \symm_n} t^{\des(w)}
\]
has the exponential generating function \cite[Theorem 1.6]{Petersen_book}
\begin{equation}
\label{des-exponential-generating-function}
F^\des(x,t) := 
1 + \sum_{n \geq 1} \frac{x^n}{n!} \symm_n^{\des}(t)
= \frac{(1-t)E}{1-t E}
\quad \text{ where } E:=e^{x(1-t)}.
\end{equation}
Using \eqref{second-cdes-gf-eqn} also for $n=0$, so that $\symm_0^\cdes(t) = 1$ implies $\symm_1^\cdes(t,u) = u$,
we wish to find an expression for
\[
\begin{aligned}
F^\cdes(x,t,u)
:=\; & xu + \sum_{n \ge 2} \frac{x^n}{n!} \symm_n^{\cdes}(t,u) \\
=\; & xu + \frac{x^2}{2!}(t+u)+ \\
  & \frac{x^3}{3!}((2t+t^2) + (1+2t)u)+ \\ 
  & \frac{x^4}{4!}((3t+8t^2+t^3) + (1+8t+3t^2)u)+ \\
  & \frac{x^5}{5!}((4t+33t^2+22t^3+t^4) + (1+22t+33t^2+4t^3)u)+ \cdots
\end{aligned}
\]
Define the integral operator $I_x[f(x)] := \int_0^x f(y)dy$.

\begin{corollary}
\[
\begin{aligned}
F^\cdes(x,t,u) 
&= \left[ xt + (u-t) \frac{\partial}{\partial t} t I_x \right] F^\des(x,t) \\
&= \frac{xt(1-t)E}{1-tE} + (u-t) \left[ \frac{(1-xt)E}{1-tE} - \frac{1}{1-t} \right]
\end{aligned}
\]
\end{corollary}
\begin{proof}
Using \eqref{second-cdes-gf-eqn},
\[
\begin{aligned}
F^\cdes(x,t,u)
&= \sum_{n \ge 1} \frac{x^n}{n!} \symm_n^{\cdes}(t,u) \\
&= \sum_{n \ge 1} \frac{x^n}{n!} \left[ nt + (u-t) \frac{\partial}{\partial t} t \right] \symm_{n-1}^\des(t) \\
&= \left[ xt + (u-t) \frac{\partial}{\partial t} t I_x \right] 
F^\des(x,t) \\
&= \left[ xt + (u-t) \frac{\partial}{\partial t} I_x t \right] 
\frac{(1-t)E}{1-tE} \\
&= \frac{xt(1-t)E}{1-tE} + (u-t) \frac{\partial}{\partial t} 
\left[ - \ln(1-tE) + \ln(1-t) \right] \\
&= \frac{xt(1-t)E}{1-tE} + (u-t) \left[ \frac{(1-xt)E}{1-tE} - \frac{1}{1-t} \right]
\end{aligned}
\]
\end{proof}

\section{Remarks and questions}
\label{sec:remarks-and-questions}

We close with several remarks and questions raised by this work.  In some cases, proofs have been
suppressed for the sake of brevity.

\subsection{Exceptional (Escher) cyclic extensions}\label{subsec:Escher}

We now explore the consequences of relaxing the non-Escher condition $\varnothing \subsetneq \cDes(T) \subsetneq [n]$.
In the context of a descent map $\Des: \TTT \longrightarrow 2^{[n-1]}$ on a finite set $\TTT$, define an {\em exceptional} (or {\em Escher}) {\em cyclic extension} of $\Des$ to be a pair $(\wcDes,p)$ satisfying the extension and equivariance axioms of Definition~\ref{def:cDes} but violating the non-Escher axiom, satisfying instead
\[
\begin{array}{rl}
\text{(Escher)}  & (\exists\, T \in \TTT)\, \wcDes(T) \in \{\varnothing, [n]\}.
\end{array}
\]

It is interesting to see how the general results on cyclic extensions $(\cDes,p)$ change in this setting.
For example, assertions (i) and (ii) of Lemma~\ref{t.cyclic_extension} still hold, with the hypothesis $m(\varnothing)=m([n])=0$ replaced by its negation $m(\varnothing)+ m([n]) > 0$.  
However, in assertion (iii), the fiber sizes 
$\wcDes^{-1}(J)$ for subsets $J=\{j_1 < \ldots < j_t\} \subseteq [n]$
are no longer uniquely determined by the fibers sizes of $\Des$;  
one also needs to know $\#\cDes^{-1}(\varnothing)$.
In fact, \eqref{cDes-fiber-sizes-formula} in assertion (iii) becomes
\begin{equation}\label{wcDes-fiber-sizes-formula}
\#\wcDes^{-1}(J) - (-1)^t  \#\wcDes^{-1}(\varnothing) =
\sum_{i=1}^{t} (-1)^{i-1} \#\Des^{-1}(\{j_{i+1} - j_i, \ldots, j_t - j_i\}).
\end{equation}
Lemma~\ref{lem:cdes-des-distribution-relation} still holds, although the generating function $\TTT^\wcdes(t)$ now has a constant term.

We shall study exceptional cyclic extensions, in some more detail, for three important examples: words, permutations, and standard Young tableaux.

\subsubsection{\bf Words}
An example of a set with a natural exceptional cyclic extension is the set $\TTT = [m]^n$ of all words $a=(a_1,\ldots,a_n)$ of length $n$ over an alphabet $[m]$.
Extending the natural definition $\Des(a):=\{i\in [n-1] \,:\, a_i > a_{i+1}\}$, define 
\[
\wcDes(a) := \{i \in [n] \,:\, a_i > a_{i+1}\} \quad
\text{ where } a_{n+1} := a_1. 
\]
All constant words $a=(j,\ldots,j)$ have $\wcDes(a)=\varnothing$.   
One can use Theorem~\ref{prop25} to prove the following.

\begin{corollary}
\label{cor:words-distribution}
For $n \geq  2$, denoting $p := \min(m,n)$, one has
\[
\begin{aligned}
\sum_{a \in [m]^n} \ttt^{\wcDes(a)}
&= m + 
\sum_{\substack{\text{\rm non-hook} \\ \lambda \vdash n}} s_\lambda(1^m) \sum_{T\in \SYT(\lambda)} \ttt^{\cDes(T)} \\
&\qquad \qquad \qquad +
\sum_{k=1}^{p-1} \sum_{t=k+1}^{p} 
   \binom{t-2}{k-1} \binom{m}{t} \binom{n-1}{t-1}
   \sum_{T \in \SYT((1^k) \oplus (n-k))} \ttt^{\cDes(T)}.
\end{aligned}
\]
\end{corollary}

\begin{proof}
Define a map $\varphi$ from $[m]^n$ to the set of all SYT of shapes which are horizontal strips of size $n$, with $m$ (possibly empty) rows, as follows: 
for each $a = (a_1, \ldots, a_n) \in [m]^n$ and $1 \le j \le m$, 
the letters in the $j$-th row (from the bottom up) of $\varphi(a)$ are $\{i \,:\, a_i = j\}$.
The map $\varphi$ clearly preserves cyclic descent sets (defined, for SYT of horizontal strip shape, as in Subsection~\ref{known-examples-subsection} above). The content vector $c = (c_1, \ldots, c_m)$ of $a$, defined by $c_j := \# \{i \,:\, a_i = j\}$ $(1 \le j \le m)$, consists of nonnegative integers summing up to $n$. Removing the entries equal to zero gives a composition $\alpha \models n$ with $1 \le \ell(\alpha) \le m$ parts. 
Noting that $\ell(\alpha) = 1 \iff a \text{ is a constant word}$, for which $\wcDes(a) = 0$, we deduce that
\[
\sum_{a \in [m]^n}
\ttt^{\wcDes(a)}
= m +
\sum_{t=2}^{m} \binom{m}{t} \sum_{\substack{\alpha \models n \\ \ell(\alpha)=t}} \sum_{T \in \SYT(\alpha^\oplus)} \ttt^{\cDes(T)}.
\]
By Theorem~\ref{prop25}, this is equal to
\[
m +
\sum_{t=2}^m \binom{m}{t} \sum_{\substack{\alpha \models n \\\ell(\alpha)=t}}
\left[ \sum_{\substack{\text{non-hook}\\ \lambda \vdash n}}
K_{\lambda,\alpha} \sum_{T \in \SYT(\lambda)} {\ttt}^{\cDes(T)}
+ \, \sum_{k=1}^{t-1} \binom{t-2}{k-1} \sum_{T\in \SYT((1^k) \oplus (n-k))} \ttt^{\cDes(T)} \right].
\]
The identities%
\footnote{To verify the first identity notice that the RHS is equal to the number of semistandard tableaux of shape $\lambda$ with letters from $[m]$, 
whereas the LHS enumerates these tableaux according to the number $t$ of letters which actually appear in the tableau ($t>1$ since $\lambda \ne (n)$) and by its content vector $\alpha$. 
}
\[
\sum_{t=2}^{m} \binom{m}{t} \sum_{\substack{\alpha \models n \\ \ell(\alpha)=t}} K_{\lambda,\alpha} = s_\lambda(1^m)
\qquad ((n) \ne \lambda \vdash n)
\]
and
\[
\# \{\alpha \models n \,:\, \ell(\alpha) = t\} 
= \binom{n-1}{t-1}
\]
complete the proof.
\end{proof}

\begin{remark}
One can define, alternatively, $\Des(a):=\{i\in [n-1] \,:\, a_i \ge a_{i+1}\}$ and extend it to
\[
\wcDes(a) := \{i \in [n] \,:\, a_i \ge a_{i+1}\} \quad
\text{ where } a_{n+1} := a_1.
\]
In that case, the constant words $a=(j,\ldots,j)$ have $\wcDes(a)=[n]$. The derivation of a suitable analogue of Corollary~\ref{cor:words-distribution} is left to the reader.
\end{remark}


\subsubsection{\bf Permutations}
A rather surprising example is the symmetric group $\TTT=\symm_n$ for $n$ even. In order to define an Escher cyclic extension of the usual $\Des$, which differs only slightly from Cellini's non-Escher extension, 
recall the definitions of {\em layered}~\cite[Definition 4.67]{Bona} and {\em colayered}~\cite{BS, ER3} permutations.

\begin{defn}
Let $1 \le k \le n$ be an integer.
A permutation $\pi \in \symm_n$ is {\em $k$-layered}
if there exist integers $0 = q_0 < q_1 < \ldots < q_{k-1} < q_k = n$
such that 
\[
\pi = \left[
q_1, q_1-1, \ldots, q_0+1, \quad
q_2, q_2-1, \ldots, q_1+1, \quad \cdots, \quad 
q_k, q_k-1, \ldots, q_{k-1}+1 
\right] ;
\] 
and $\pi\in \symm_n$ is {\em $k$-colayered} if the reverse partition $\pi^r = (\pi(n), \ldots, \pi(1))$ is $k$-layered.
\end{defn}

\begin{defn}\label{def:wcDes_for_even_symm}
For even $n$, let $\cDes : \symm_n \to 2^{[n]}$ be Cellini's cyclic descent map, as in \eqref{e.cellini}.
Define  $\wcDes : \symm_n \to 2^{[n]}$ as follows: 
\[
\wcDes(\pi):=\begin{cases}
\cDes(\pi) \setminus \{n\} = \Des(\pi), 
& \text{ if $\pi$ is $k$-layered with $k$ even} ;\\
\cDes(\pi) \sqcup \{n\} = \Des(\pi) \sqcup \{n\}, 
& \text{ if $\pi$ is $k$-colayered with $k$ even} ;\\
\cDes(\pi), 
& \text{ otherwise.}
\end{cases}
\]
\end{defn}

\begin{example}
For $n=2$, the identity permutation $12$ is 2-layered (and 1-colayered) while $21$ is 2-colayered (and 1-layered); by definition, $\wcDes(12)=\varnothing$ while $\wcDes(21)=\{1,2\}$. 
For $n=4$, 
the even-layered permutations are $1234, 1432, 2143$ and $3214$,
with corresponding $\wcDes$-values
$\varnothing, \{2,3\}, \{1,3\}$ and $\{1,2\}$. 
The even-colayered permutations are $4321,4123,3412$ and $2341$,
with corresponding $\wcDes$-values
$\{1,2,3,4\}, \{1,4\}, \{2,4\}$ and $\{3,4\}$.
\end{example}

It is not hard to prove the following.

\begin{proposition}\label{t:symm_exceptional}
For even $n$, Definition~\ref{def:wcDes_for_even_symm} gives an exceptional cyclic extension $(\wcDes,p)$ of the usual descent map $\Des$ on $\symm_n$, with the same cyclic map $p$ as in Cellini's extension $(\cDes,p)$.
The fiber sizes satisfy
\begin{equation}
\label{eq:nonstandard}
\#\wcDes^{-1}(J) 
=  \#\cDes^{-1}(J) + (-1)^{\#J} \qquad (\forall J \subseteq [n]).
\end{equation}
In particular, $\#\wcDes^{-1}(\varnothing) = \#\wcDes^{-1}([n]) = 1$. 
\end{proposition}


Combining Equation \eqref{eq:nonstandard} with Corollaries~\ref{cor:fiber-sizes-as-inner-product}
and~\ref{affine-ribbon-hook-pairing}, one deduces

\begin{corollary}
	\[
	\#\wcDes^{-1}(J) 
	= \langle  s_{(1^n)^\oplus}-s_n, \cs_{\ccomp{n}{J}} \rangle
	\qquad (\forall\ \varnothing \ne J \subseteq [n]).
	\]	
\end{corollary}

Note that this also implies that, for $n$ even,
\[
\symm_n^{\wcDes}(\ttt) - \symm_n^\cDes(\ttt)
= \prod\limits_{i=1}^n (1-t_i).
\]
On $\symm_n$, no exceptional cyclic extensions exist for odd $n > 1$, and all exceptional cyclic extensions have the same fiber sizes for even $n$.
These claims follow from 
Theorem~\ref{conj1_Escher} below.

\subsubsection{\bf Standard Young tableaux}
Our main focus in this paper is on standard Young tableaux. For them,  an ``exceptional'' analogue of Theorem~\ref{conj1} is the following.
\begin{theorem}\label{conj1_Escher}
Let $\lambda/\mu$ be a skew shape of size $n \ge 2$.
The usual descent map $\Des$ on $\SYT(\lambda/\mu)$ has an exceptional cyclic extension $(\wcDes,p)$ if and only if $\lambda/\mu$ has one of the following forms. In each case, all the exceptional cyclic extensions share the same cardinalities $\#\wcDes^{-1}(J)$, for all $J \subseteq [n]$.
\begin{enumerate}
\item
$\lambda/\mu$ has a single row. In that case, $\wcDes(T) = \varnothing$ for the unique $T \in \SYT(\lambda/\mu)$.
\item
$\lambda/\mu$ has a single column. In that case, $\wcDes(T) = [n]$ for the unique $T \in \SYT(\lambda/\mu)$.
\item
The size $n$ is even and $\lambda/\mu$ has $n$ connected components, each of size $1$. $\SYT(\lambda/\mu)$ has a natural bijection with the symmetric group $\symm_n$, and the unique value distribution of exceptional cyclic extensions is given by Proposition \ref{t:symm_exceptional} above.
\end{enumerate}
\end{theorem}

\begin{remark}
For $n=1$, the unique $T \in \SYT(\lambda/\mu)$ has $\Des(T) = \varnothing$. In this case there are two distinct exceptional cyclic extensions, one with $\wcDes(T) = \varnothing$ and the other with $\wcDes(T) = [1]$.
\end{remark}

\begin{proof}[Proof of Theorem~\ref{conj1_Escher}.]
Existence, in the last case, follows from Proposition~\ref{t:symm_exceptional}, and is obvious for the other cases.
It remains to show uniqueness, including the claim that an exceptional cyclic extension does not exist in any other case.

Indeed, assume that $(\wcDes,p)$ is an exceptional cyclic extension of the usual descent map $\Des$ on $\SYT(\lambda/\mu)$, for a skew shape $\lambda/\mu$ of size $n \ge 2$. Assume that there is a tableau $T_0 \in \SYT(\lambda/\mu)$ with $\wcDes(T_0) = \varnothing$;
the treatment of the case $\wcDes(T_0) = [n]$ is analogous and is left to the reader. Of course, necessarily $\Des(T_0) = \varnothing$ by the extension axiom, and this implies that each connected component of $\lambda/\mu$ consists of a single row.
If there is only one connected component, we get the first case of the theorem. Assume, therefore, that there are $m \ge 2$ connected components.

For each $1 \le t \le n$, consider the set of all $T \in \SYT(\lambda/\mu)$ with $\Des(T) = [t-1]$. In such a tableau, the entries $1, \ldots, t$ appear in distinct components (rows), in descending order, with entry $t$ at the last (southernmost) row. Each of these entries occupies the first (westernmost) cell in its row. The rest of the tableau must be filled in a unique fashion, and it follows that
\[
\#\Des^{-1}([t-1]) = \binom{m-1}{t-1} \qquad (1 \le t \le n).
\]
of course, this number is zero unless $1 \le t \le m$.
It follows from \eqref{wcDes-fiber-sizes-formula} that
\[
\begin{aligned}
\#\wcDes^{-1}([t]) - (-1)^t \#\wcDes^{-1}(\varnothing) 
&= \sum_{i=1}^{t} (-1)^{i-1} \#\Des^{-1}([t-i]) \\
&= \sum_{j=1}^{t} (-1)^{t-j} \#\Des^{-1}([j-1])
= 0 \qquad (\forall t \ge m).
\end{aligned}
\]
In particular, for $t = n$,
\[
\#\wcDes^{-1}([n]) = (-1)^n \#\wcDes^{-1}(\varnothing).
\]
Together with our assumption $\#\wcDes^{-1}(\varnothing) >0$, this
implies that $n$ is even; and if $m <n$ then, for $t = n-1$,
\[
\#\wcDes^{-1}([n-1]) = (-1)^{n-1} \#\wcDes^{-1}(\varnothing) < 0
\]
gives a contradiction. Thus $m=n$, each component consists of a single cell, and we are in the third case of the theorem. 
Since $\#\Des^{-1}(\varnothing) = 1$ and $\#\wcDes^{-1}(\varnothing) > 0$, necessarily $\#\wcDes^{-1}(\varnothing) = 1$.
\eqref{wcDes-fiber-sizes-formula} shows that the distribution of $\wcDes$ values is unique.
\end{proof}

\subsection{Topological interpretation of affine ribbon Schur functions}
\label{Steinberg-torus-section}

The alternating sum definition \eqref{affine-ribbon-Schur-function} of the affine ribbon Schur function 
$\cs_{\ccomp{n}{J}}:= \sum_{\varnothing \neq I \subseteq J}
(-1)^{\#(J \setminus I)} h_{\ccomp{n}{I}}$ has a topological interpretation,
as the (Frobenius image of) a certain virtual Euler characteristic representation of $\symm_n$.
In particular, the special case $t=n$ of \eqref{e.cyclic-ribbon-expansion},
\[
\label{Steinberg-torus-Euler-Poincare}
\sum_{\varnothing \neq I \subseteq [n]} (-1)^{n-\#I}h_{\ccomp{n}{I}} = \cs_{\ccomp{n}{[n]}} = \sum_{i=0}^{n-1} (-1)^{n-1-i} s_{(n-i,1^i)},
\]
is the {\em Euler-Poincar\'e relation} for the (ordinary, non-reduced) homology of the (type $A_{n-1}$) {\em Steinberg torus} considered in \cite{DPS}.

We first recall the known topological interpretation for the ribbon skew Schur function in terms of the {\em type $A_{n-1}$ Coxeter complex}.
This is a simplicial complex $\Delta$ triangulating an $(n-2)$-dimensional sphere, with a simply transitive action of $\symm_n$ on its maximal simplices.
It also has a {\em balanced} coloring of its vertices by $[n-1]$:  each maximal simplex has exactly one vertex of each color.
Furthermore, for each $J \subseteq [n-1]$, the group $\symm_n$ acts transitively on the simplices whose vertices have color set $J$,
but now with $\symm_n$-stabilizers conjugate to the Young subgroup $\symm_{\comp{n}{J}}$ associated to the composition $\comp{n}{J}$.  Thus the permutation representation of $\symm_n$ on simplices of color set $J$ has image, under the {\it Frobenius characteristic map} $\ch$, equal to the symmetric function $h_{\comp{n}{J}}$.  
This completely describes the action of $\symm_n$ on the simplices of $\Delta$.

On the homological side, since $\Delta$ triangulates a sphere, it is {\em Cohen-Macaulay}, and hence has only top-dimensional (reduced) homology.
This Cohen-Macaulay property is also inherited, for each subset $J \subseteq [n-1]$, by the {\em type-selected subcomplex} $\Delta_J$ consisting of the simplices that only use vertices whose colors lie in $J$. 
The Euler-Poincar\'e relation for $\Delta_J$ says that
\[
\sum_{i \geq -1} (-1)^i \ch( \tilde{H}_i(\Delta_J) )
= \sum_{i \geq -1} (-1)^i \ch( \tilde{C}_i(\Delta_J) )
\]
where $\tilde{C}_i, \tilde{H}_i$ are (augmented/reduced) chain
and homology groups, taken with rational coefficients.
By the above discussion this gives
\[
(-1)^{\#J-1} \ch ( \tilde{H}_{\#J-1}(\Delta_J) )
=\sum_{I \subseteq J} (-1)^{\#I -1} h_{\comp{n}{I}}
\]
Re-writing this last line gives a well-known homological re-interpretation \cite[\S 6]{Bjorner}, \cite{Solomon}, \cite[\S 4]{Stanley} of $s_{\comp{n}{J}}$:
\[
\ch (\tilde{H}_{\#J-1}(\Delta_J))
=\sum_{I \subseteq J} (-1)^{\#(J \setminus I)} h_{\comp{n}{I}}
= s_{\comp{n}{J}}.
\]

We wish to similarly re-interpret, for $\varnothing \neq J \subseteq [n]$, the affine ribbon skew Schur function $\cs_{\ccomp{n}{J}}$
in terms of the type $A_{n-1}$ case of what Dilks, Petersen and Stembridge \cite{DPS} call the {\em Steinberg torus}.
This is a regular cell complex which we shall denote $\SteinbergTorus$.
It is a {\em Boolean cell complex}: all lower intervals in the partial ordering of cells are Boolean algebras, so that the cells are essentially simplices, but their intersections are not necessarily common faces.
It triangulates an $(n-1)$-dimensional torus, with a simply transitive action of $\symm_n$ on the maximal cells.
It also has a balanced coloring of its vertices by $[n]$, so each maximal simplex has exactly one vertex of each color.
Furthermore, for $\varnothing \neq J \subseteq [n]$, the group $\symm_n$ again acts transitively on the set of all cells whose vertices have color set $J$,
but this time with $\symm_n$-stabilizers conjugate to the Young subgroup $\symm_{\ccomp{n}{J}}$ associated to the {\em cyclic} composition $\ccomp{n}{J}$.  
Thus the permutation representation of $\symm_n$ on the cells of color set $J$ has image, under the Frobenius characteristic map $\ch$, equal to the symmetric function $h_{\ccomp{n}{J}}$.  This
describes the action of $\symm_n$ on the simplices of $\SteinbergTorus$.

On the homological side,
since $\SteinbergTorus$ triangulates a torus, it is not Cohen-Macaulay.
In fact, its (non-reduced) cohomology ring $H^*(\SteinbergTorus)$ 
with rational coefficients is isomorphic to an exterior algebra $\wedge V$, where $V = H^1(\SteinbergTorus)$ carries the irreducible reflection representation of $\symm_n$.  Since
$\wedge^i V$ has Frobenius image $\ch(\wedge^i V)=s_{(n-i,1^i)}$,
this describes the $\symm_n$-action on homology.
An analysis via the Euler-Poincar\'e relation, as before, shows that
\[
\sum_{i \geq 0} (-1)^i \ch( C^i(\SteinbergTorus) ) 
= \sum_{i \geq 0} (-1)^i \ch( H^i(\SteinbergTorus) )
\]
which becomes exactly \eqref{Steinberg-torus-Euler-Poincare} by
the above discussion.  
More generally, for each $\varnothing \neq J \subseteq [n]$, the {\em type-selected subcomplex} $\SteinbergTorus_J$ consisting of the cells that only use vertices whose colors lie in $J$ has Euler-Poincar\'e relation
\[
\sum_{i \geq 0} (-1)^i \ch( C^i(\SteinbergTorus_J) ) 
= \sum_{i \geq 0} (-1)^i \ch( H^i(\SteinbergTorus_J) )
\]
giving the re-interpretation
\[
\cs_{\ccomp{n}{J}}
=\sum_{\varnothing \neq I \subseteq J} (-1)^{\#(J \setminus I)} h_{\ccomp{n}{I}}
=\sum_{i \geq 0} (-1)^{\#J-1 - i} \ch (H^i(\SteinbergTorus_J)).
\]



\subsection{Bijective proofs and cyclic sieving}

\subsubsection{\bf Bijective proofs and dihedral group action}
Our proof of the existence of $(\cDes,p)$ in Theorem~\ref{conj1}
is indirect and involves arbitrary choices. It is desired to have a constructive proof,
which will provide an explicit combinatorial definition of the cyclic descent set map.

\begin{problem}
	\label{natural-maps-problem}
	Find a natural, explicit map $\cDes$ and cyclic action $p$ on $\SYT(\lambda/\mu)$ as in Theorem~\ref{conj1}.
\end{problem}

For discussions and solutions for specific shapes, see~\cite{Rhoades, Pechenik, DPS2, AER}.


\medskip

Anders Bj\"orner suggested the problem of finding an explicit dihedral group action
on $\SYT(\lambda/\mu)$ with nice properties. 
More precisely,
recall from Corollary~\ref{cor:fiber-sizes-involution} that the cyclic descent set and its negative are equidistributed over the SYT of any given non-ribbon skew shape.

\begin{problem}\label{involution-problem}
Given a solution of Problem~\ref{natural-maps-problem}, 
find an 
appropriate involution $\iota$ on $\SYT(\lambda/\mu)$ which sends the cyclic descent set to its negative.
\end{problem}

\noindent
One wants $\iota$ to interact well with an explicit cyclic map $p$ which shifts 
the cyclic descent sets, so that they satisfy the relation $\iota p \iota = p^{-1}$.
Letting $\iota$ and $p$ 
be 
evacuation 
and jeu-de-taquin promotion respectively provides a solution for rectangular shapes
~\cite{Rhoades}. 
The general case is wide open.

\medskip

The proof of Theorem~\ref{conj2} is also indirect.

\begin{problem}\label{Schensted-problem}
Find a Robinson-Schensted-style bijective proof of Theorem~\ref{conj2}.
\end{problem}

\subsubsection{\bf Cyclic sieving phenomenon?}
Rhoades \cite{Rhoades} proved that,
for rectangular shapes $\lambda$,
the usual {\em jeu-de-taquin} promotion operator $p: \SYT(\lambda) \rightarrow \SYT(\lambda)$
has order $n:=|\lambda|$ and, for any $k$,
\[
\# \{T \in \SYT(\lambda): p^k(T)=T \}
= \left[ f^\lambda(q) \right]_{q=\zeta^k}\
\]
where $\zeta:=e^{\frac{\pi i}{n}}$, and
\[
f^\lambda(q) 
:= \frac{[n]!_q}{\prod_{x \in \lambda} [h(x)]_q}
= q^{-b(\lambda)} \sum_{T \in \SYT(\lambda)} q^{\operatorname{maj}(T)}
\]
is the usual {\em $q$-hook formula} \cite[Cor. 7.21.5]{EC2}\footnote{
An elegant refinement of this CSP
on rectangular shapes was  conjectured, and most recently proved for two-row shapes, by Ahlbach, Rhoades and Swanson~\cite{ARS}.}.

\begin{problem}
\label{CSP-problem}
For non-hook shapes $\lambda$ besides rectangles,
can one choose the operator $p$ in Theorem~\ref{conj1}
and a polynomial $X(q)$ to replace $f^\lambda(q)$
so that this cyclic sieving phenomenon (CSP)  generalizes?
\end{problem}

Unfortunately, $f^\lambda(q)$ itself will not always work.

\begin{example}
Take $\lambda=(3,2,1)$ as in Example~\ref{non-unique-orbits-example}.
Unfortunately, no matter how one chooses the orbit structure in this example, if one plugs $q=\zeta^2$
with $\zeta:=e^{\frac{2 \pi i}{6}}$ into
$$
f^\lambda(q)= 1+2q+2q^2+3q^3+3q^4+2q^5+2q^6+q^7
$$
one obtains $2(1+\zeta^2)$, instead of a (nonnegative) integer.
\end{example}

On the other hand, the usual jeu-de-taquin promotion operator $p_{jdt}$ on
$\SYT(\lambda)$ is known to have order $N(N-1)=2|\lambda|$
when $\lambda=(N-1,N-2,\ldots,2,1)$, and hence  $p^2_{jdt}$ has order
$|\lambda|$.  Thus one might ask whether
the action of $p^2_{jdt}$ has orbit sizes related to those of $p$
from the cyclic extension $(\cDes,p)$.  For $N=4$, the orbit sizes
of $p_{jdt}$ on $\SYT((3,2,1))$ are $12$ and $4$, so the
orbit sizes of $p^2_{jdt}$ are $6,6,2,2$, which can be consistent with the
$p$-orbits described above.  Unfortunately, neither $p_{jdt}$ nor
$p^2_{jdt}$ makes $\cDes$ equivariant in this case. 


\medskip


A cyclic action on SYT of shape $(k,k,1^{n-2k})$ and a corresponding CSP were introduced by Pechenik~\cite{Pechenik}; see also~\cite{DPS2} and~\cite[\S 2.8]{Iraci}.
This result may be used to define an explicit cyclic descent extension for SYT of this shape.

\medskip

Finally,  
recalling from~\cite{ER1} the cyclic descent extension for $\SYT(\lambda \oplus (1))$, 
Corollary~\ref{cor:fiber-sizes-as-inner-product} in the current paper has been applied 
by Ahlbach, Rhoades, and Swanson in~\cite{ARS}
to obtain a
refined CSP on SYT of these skew shapes.

\begin{thebibliography}{99}

\bibitem{AER} 
R.\ M.\ Adin, S.\ Elizalde and Y.\ Roichman,
Cyclic descents for standard Young tableaux of two-row shapes,
preprint 2016.

\bibitem{ARS} C.\ Ahlbach, B.\ Rhoades and J.\ P.\ Swanson,
{\em Euler-Mahonian refined cyclic sieving}, in preparation.

\bibitem{BS} 
C.\ Benedetti and B.\ E.\ Sagan,
Antipodes and involutions, 
{\em J.\ Combin.\ Theory Ser.\ A}
{\bf 148} (2017), 275--315.

\bibitem{Bjorner}
A.\ Bj\"orner,
Some combinatorial and algebraic properties of Coxeter complexes and Tits buildings,
{\em Adv.\ in Math.} {\bf 52} (1984), 173--212.

\bibitem{Bona} 
M.\ B\'ona, 
Combinatorics of Permutations, 
2nd edition. Discrete Mathematics and its Applications, CRC Press, Boca Raton, 2012.

\bibitem{BKPT}
A.\ S.\ Buch, A.\ Kresch, K.\ Purbhoo and H.\ Tamvakis,
The puzzle conjecture for the cohomology of two-step flag manifolds,
{\em J.\ Algebraic Combin.} {\bf  44} (2016), 973–-1007.

\bibitem{Cellini}
P.\ Cellini,
Cyclic Eulerian elements,
{\em Europ.\ J.\ Combinatorics} {\bf 19} (1998), 545--552.

\bibitem{CPY}
M.\ Chmutov, P.\ Pylyavskyy and E.\ Yudovina,
Matrix-Ball Construction of affine Robinson-Schensted correspondence,
{\tt arXiv:1511.05861}.

\bibitem{DPS2} 
K.\ Dilks, O.\ Pechenik and J.\ Striker,   
Resonance in orbits of plane partitions and increasing tableaux, 
{\em J.\ Combin.\ Theory Ser.\ A} {\bf 148} (2017), 244--274. 

\bibitem{DPS}
K.\ Dilks, K.\ Petersen and J.\ Stembridge,
Affine descents and the Steinberg torus,
{\em Adv.\ in Applied Math.} {\bf 42} (2009), 423--444.

\bibitem{ER3}
S.\ Elizalde and Y.\ Roichman,
Schur-positive sets of permutations via products of grid classes,
{\em J.\ Algebraic Combin.} {\bf  45} (2017), 363-–405.

\bibitem{ER1} 
\bysame,
On rotated Schur-positive sets, 
{\em J.\ Combin.\ Theory Ser.\ A} {\bf 152} (2017), 121–-137.

\bibitem{Fulman} 
J.\ Fulman, 
Affine shuffles, shuffles with cuts, the Whitehouse module, and patience sorting,
{\em J.\ Algebra} {\bf 231} (2000), 614--639.

\bibitem{Gessel}
I.\ M.\ Gessel,
Multipartite P-partitions and inner products of skew Schur functions, in: Combinatorics and algebra (Boulder, Colo., 1983), pp.\ 289--317,
{\em Contemp. Math.} {\bf 34},
Amer.\ Math.\ Soc., Providence, RI, 1984.

\bibitem{GesselKrattenthaler}
I.\ M.\ Gessel and C.\ Krattenthaler,
Cylindric partitions,
{\em Trans.\ Amer.\ Math.\ Soc.} {\bf  349} (1997), 429--479.

\bibitem{Iraci} A.\ Iraci, Cyclic sieving phenomenon
in noncrossing partitions, M.Sc. Thesis, Univ. di Pisa, 2016.

\bibitem{Kane}
R.\ Kane,
Reflection groups and invariant theory,
{\em CMS Books in Mathematics} {\bf 5}.,
Springer-Verlag, New York, 2001.

\bibitem{Macdonald}
I.\ G.\ Macdonald,
Symmetric functions and Hall polynomials, 
2nd edition, {\em Oxford Classic Texts in the Physical Sciences},
The Clarendon Press, Oxford University Press, New York, 2015. 

\bibitem{McNamara}
P.\ McNamara,
Cylindric skew Schur functions,
{\em Adv.\ Math.} {\bf 205} (2006), 275--312

\bibitem{MorseSchilling}
J.\ Morse and A.\ Schilling, 
Crystal approach to affine Schubert calculus,
{\em Int.\ Math.\ Res.\ Not.} (2016),  2239--2294.

\bibitem{Pawlowski}
B.\ Pawlowski,
A representation-theoretic interpretation of positroid classes,
{\tt arXiv:1612.00097}.

\bibitem{Pechenik}
O.\ Pechenik, 
Cyclic sieving of increasing tableaux and small Schröder paths,
{\em J.\ Combin.\ Theory Ser.\ A} {\bf 125} (2014), 357–-378.

\bibitem{Petersen_paper}
T.\ K.\ Petersen,
Cyclic descents and $P$-partitions,
{\em J.\ Algebraic Combin.} {\bf 22} (2005), 343–-375.

\bibitem{Petersen_book}
\bysame,
Eulerian numbers,
{\em Birkh\"auser Advanced Texts},
Birkh\"auser/Springer, New York, 2015.

\bibitem{Postnikov}
A.\ Postnikov,
Affine approach to quantum Schubert calculus,
{\em Duke Math.\ J.} {\bf 128} (2005), 473--509.

\bibitem{Rhoades}
B.\ Rhoades,
Cyclic sieving, promotion, and representation theory
{\em J.\ Combin.\ Theory Ser.\ A} {\bf 117} (2010) 38--76.

\bibitem{Sagan}
B.\ E.\ Sagan, 
The symmetric group: Representations, combinatorial algorithms, 
and symmetric functions,
2nd edition,
{\em Graduate Texts in Mathematics} {\bf 203}. Springer-Verlag, New York, 2001.

\bibitem{Solomon}
L.\ Solomon,
A decomposition of the group algebra of a finite Coxeter group,
{\em J.\ Algebra} {\bf 9} (1968), 220--239.

\bibitem{Stanley}
R.\ P.\ Stanley,
Some aspects of groups acting on finite posets,
{\em J.\ Combin.\ Theory Ser.\ A} {\bf 32} (1982), 132--161.

\bibitem{EC1}
\bysame,
Enumerative combinatorics, Vol.~1, Second edition,
{\em Cambridge Studies in Advanced Mathematics}~{\bf 49},
Cambridge Univ.\ Press, Cambridge, 2012.

\bibitem{EC2}
\bysame,
Enumerative combinatorics, Vol.~2, {\em Cambridge Studies in Advanced Mathematics}~{\bf  62}, Cambridge Univ.\ Press, Cambridge, 1999.

\end{thebibliography}
\end{document}